\documentclass{article}

\usepackage[utf8]{inputenc}
\usepackage[normalem]{ulem} 	
\usepackage{xparse}         	
\usepackage{comment}

\usepackage{mathtools,amssymb,amsthm,empheq}
\usepackage{bm} 	   	
\usepackage{relsize}   	
\usepackage{mathrsfs}  	
\usepackage{cancel}   	
\usepackage{esint}    	
\usepackage{xcolor} 	
\usepackage{cases}

\usepackage{pgf,tikz}
\usetikzlibrary{cd,arrows,babel,shapes,positioning}

\usepackage{enumerate}

\usepackage{hyperref}
\hypersetup{
	colorlinks,
	citecolor=blue,
	filecolor=blue,
	linkcolor=blue,
	urlcolor=blue
}

\DeclareMathOperator{\Div}{div}

\newcommand{\AP}{A_{\Vert}}

\newcommand{\ColorWord}[2]{\color{#1} #2 \color{black} }

\allowdisplaybreaks

\numberwithin{equation}{section}

\theoremstyle{plain}

\newtheorem{thm}[equation]{Theorem}
\newcommand{\refthm}[1]{\emph{\ColorWord{blue}{Theorem} \ref{#1}}}

\newtheorem{lemma}[equation]{Lemma}
\newcommand{\reflemma}[1]{\emph{\ColorWord{blue}{Lemma} \ref{#1}}}

\newtheorem{prop}[equation]{Proposition}
\newcommand{\refprop}[1]{\emph{\ColorWord{blue}{Proposition} \ref{#1}}}

\newtheorem{cor}[equation]{Corollary}
\newcommand{\refcor}[1]{\emph{\ColorWord{blue}{Corollary} \ref{#1}}}

\theoremstyle{definition}

\newtheorem{defin}[equation]{Definition}
\newcommand{\refdef}[1]{\emph{Definition \ref{#1}}}

\theoremstyle{remark}

\newtheoremstyle{named}{}{}{\itshape}{}{\bfseries}{}{.5em}{#1 #3}
\theoremstyle{named}
\title{The regularity problem with a weaker condition on only the transversal direction}
\author{Martin Ulmer}
\date{\today}

\begin{document}

\maketitle

\begin{abstract}
\noindent
We study an elliptic operator \(L:=\mathrm{div}(A\nabla \cdot)\) on the upper half space. It is known that if the matrix \(A\) is independent in the transversal \(t\)-direction, then the regularity boundary value problem is solvable with data in a Sobolev space. In the present paper we improve on the \(t\)-independence condition by introducing a mixed \(L^1-L^\infty\) condition that only depends on \(\partial_t A\), the derivative of \(A\) in transversal direction. This condition is different from other conditions in the literature and has already been proven to imply solvability of the Dirichlet boundary value problem.
\end{abstract}

\tableofcontents

\section{Introduction}

In this work let \(\Omega:=\mathbb{R}^{n+1}_+:=\mathbb{R}^n\times (0,\infty), n\geq 1\) be the upper half space and \(\mathcal{L}:=\Div(A\nabla\cdot)\) an uniformly elliptic operator with bounded measurable coefficients, i.e. \(A(x,t)\) is a real not necessarily symmetric \(n+1\) by \(n+1\) matrix and there exists \(\lambda_0>0\) such that
\begin{align}
    \lambda_0 |\xi|^2\leq \xi^T A(x,t) \xi \leq \lambda_0^{-1}|\xi|^2 \qquad \textrm{ for all }\xi\in \mathbb{R}^{n+1},\label{eq:DefinitionOfUniformElliptic}
\end{align}
and a.e. \((x,t)=(x_1,...,x_n,t)\in \mathbb{R}^{n+1}_+\). We are interested in the solvability of the regularity boundary value problem 
\[\begin{cases} \mathcal{L}u=\Div(A\nabla u)=0 &\textrm{in }\Omega, \\ u=f &\textrm{on }\partial\Omega,\end{cases}\]
with boundary data \(f\) in the Sobolev space \(W^{1,p}(\partial\Omega)\) (see \refdef{def:L^pregularityProblem}). In contrast to the Dirichlet boundary value problem with boundary data in \(L^p\), the regularity problem also captures the behavior of the tangential derivative of the solution \(u\) on the boundary.
\medskip

The question we would like to pose is quite well established and goes back to Dahlberg (cf. \cite{dahlberg_estimates_1977} and \cite{dahlbert_poisson_1979}). He showed the existence of solutions to the Dirichlet boundary value problem for the Laplacian and boundary data in \(L^2\) on Lipschitz graph domains. More interestingly however, the following observation appeared: 
It turns out that finding harmonic functions with \(L^2\) boundary data on a Lipschitz graph domain \(\Omega\) for
\[\begin{cases} \Delta u=0 &\textrm{in }\Omega\\ u=f &\textrm{on } \partial\Omega\end{cases}\qquad \textrm{is equivalent to solving } \qquad \begin{cases} Lu=0 &\textrm{in }\mathbb{R}^n\times(0,\infty)\\  u= f &\textrm{on } \mathbb{R}^n\end{cases}\]
for an elliptic operator \(L\) depending on the Lipschitz function of \(\Omega\). This means that solving an easy elliptic PDE on a harder Lipschitz graph domain is equivalent to solving a harder elliptic PDE on the easier domain of the upper half space.
Depending on which flattening one uses in the argument, the operator \(L\) does have some additional properties, and the most established ones are the Carleson condition (or also Dahlberg-Kenig-Pipher (DKP) condition) and the \(t-\)independence condition. The literature regarding the DKP condition is rich and not exclusive to the Dirichlet boundary value problem, but finds applications in a range of different elliptic or parabolic boundary value problems (cf. the survey article \cite{dindos_boundary_2022} and references therein). Similarly, the \(t-\)independent condition gave rise to solvability of other boundary value problems (cf. \cite{kenig_new_2000}, \cite{kenig_regularity_2008}, \cite{hofmann_regularity_2015}, \cite{hofmann_regularity_2015}, \cite{castro_boundedness_2016}, \cite{nystrom_l2_2017}, \cite{auscher_dirichlet_2018}, \cite{hofmann_dirichlet_2022}). Broadly speaking, the DKP condition controls the Lipschitz constant of the matrix \(A\) close to the boundary and asks it to grow less than the function \(1/t\) close to the boundary. The \(t-\)independence condition on the other hand requires \(A(x,t)=A(x)\), i.e. that \(A\) is independent in transversal direction with potentially almost arbitrarily bad behavior in any direction parallel to the boundary.
\\

Both of these structurally different conditions, the DKP and \(t\)-independent condition, are sufficient conditions for solvability of the boundary value problems, and hence we can wonder whether there are other sufficient conditions or improvements of these two, that still allow us to solve boundary value problems. Finding improvements will narrow the amount of operators for which we do not know whether any boundary value problem is solvable. It is also noteworthy that \cite{modica_construction_1980} and \cite{caffarelli_completely_1981} provide examples of linear uniformly elliptic operators for which the Dirichlet boundary value problem is not solvable. Furthermore, the article \cite{kenig_new_2000} demonstrates that if the matrix \(A\) is nonsymmetric, we cannot expect to obtain \(L^2\) solvability for the Dirichlet problem with \(t-\)independent coefficients, whence the best we can hope for is solvability with \(L^p\) data for potentially large \(p\). Since the solvability range of the regularity problem is dual to that of the Dirichlet problem, in our case we can only expect solvability for potentially small \(p>1\).
\medskip

The extension from symmetric to nonsymmetric matrices in the Dirichlet and regularity boundary value problem required tools provided by the resolved Kato conjecture (cf. \cite{Auscher_Kato}) which was previously open for almost 50 years. Despite that the \(t-\)independent condition for the Dirichlet problem in the symmetric case was well known since Jerison and Kenig used a ``Rellich" identity in \cite{jerison_dirichlet_1981}, it took the Kato tools until the authors of \cite{hofmann_square_2015} could prove the nonsymmetric case.
\medskip

Apart from the above two mentioned conditions there is also the Dini-condition from \cite{fabes_necessary_1984}, where they showed that \(t\)-independence can be relaxed if we use continuous coefficients. More precisely, they show that if a symmetric \(A\) has continuous coefficients, \(\Omega\) is a bounded \(C^1\)-domain, and the modulus of continuity 
\[\eta(s)=\sup_{P\in \partial\Omega, 0<r<s}|A_{ij}(P-rV(P))-A_{ij}(P)|\]
with outer normal vector field \(V\) satisfies the Dini-type condition
\begin{align}\int_0\frac{\eta(s)^2}{s}ds<\infty,\label{DiniTypeCond}\end{align}
then \(\omega\in B_2(\sigma)\subset A_\infty(\sigma)\). Together with the counterexample in \cite{jerison_dirichlet_1981}, where completely singular measures \(\omega\) with respect to the surface measures are constructed for a given \(\eta\) with \(\int_0 \frac{\eta(s)^2}{s}ds=+\infty\), the Dini type condition (\ref{DiniTypeCond}) turns out to be sufficient for \(\omega\in A_\infty(d\sigma)\) and kind of necessary, if \(A\) is symmetric with continuous bounded coefficients. A little bit later, \cite{dahlberg_absolute_1986} extended this result to symmetric matrices with merely bounded and measurable coefficients by considering perturbations from continuous matrices.
\medskip

The condition we would like to study in the article in the mixed \(L^1-L^\infty\) condition that appears in \cite{ulmer_mixed_2024}. The condition is stated as
\begin{align}
\int_0^\infty \Vert \partial_t A(\cdot, t)\Vert_{L^\infty}dt<\infty. \label{cond:mixedL1LInftyCond}
\end{align}
From \cite{ulmer_mixed_2024} we know that if \(|\partial_t A|\leq C/t\) and \eqref{cond:mixedL1LInftyCond} is satisfied, then the Dirichlet problem is solvable for some \(p<\infty\). It is easy to see that this condition is a generalization of the \(t-\)independent condition, but it is also different from the Dini-condition. Even if the Dini-condition holds true as sufficient condition for unbounded domains and nonsymmetric matrices and the regularity problem, which is not established yet, we can find examples of matrices in \cite{ulmer_mixed_2024} that satisfy either condition but not both. Interestingly for \(n=1\), the same methods also yield an improvement. In the upper half plane, we only need to assume an \(L^1\)-Carleson condition on \(|\partial_t A|\) instead of \eqref{cond:mixedL1LInftyCond} (cf. \cite{ulmer_solvability_2025}).
\medskip

Since these new conditions yield solvability for the Dirichlet problem it is natural to ask whether they also allow us to solve the regularity problem. This is our main result

\begin{thm}\label{MAINTHM}
    Assume \(L:=\mathrm{div}(A\nabla \cdot)\) is uniformly elliptic operator with bounded, merely measurable coefficients and let \(\Omega=\mathbb{R}^{n+1}_+\). If \(A\) satisfies \(|\partial_t A|\leq C/t\) and \eqref{cond:mixedL1LInftyCond} then there exists \(p>1\) such that \eqref{eq:RPinequalityElliptic} holds and hence the regularity boundary value problem is solvable for \(f\in \dot{L}^p_1(\partial\Omega)\).
\end{thm}

Since the mixed \(L^1-L^\infty\) condition implies the \(t-\)independent condition, the presented proof is also an alternative proof of solvability of the regularity problem for \(t\)-independent operators. This result has been proved in \cite{hofmann_regularity_2015} and our proof only relies on integration by parts techniques and all the tools the resolved Kato conjecture provides including the semi group theory, but does not need single layer potentials like \cite{hofmann_regularity_2015}. To prove the mixed \(L^1-L^\infty\) condition we will have to bound area function estimates involving the operator \(\mathcal{P}_t:=e^{-t^2L_{||}^t}\) but if we only want to see the result for \(t-\) independent operators, we can replace this operator by the same \(\mathcal{P}_t:=e^{-t^2L_{||}}\) that is used in \cite{hofmann_regularity_2015}. Then all the required estimates on \(\mathcal{P}_t\) are already given and used in \cite{hofmann_regularity_2015}, and our proof is an alternative to the proof presented in \cite{hofmann_regularity_2015}.
\medskip

The main part through chapters 3 to 5 consists of proving \(L^p\) boundedness of the area function involving terms like \(\mathcal{P}_t, \partial_t\mathcal{P}_t\), or \(\nabla_{||}\mathcal{P}_t\). Recall that for the Dirichlet problem in \cite{ulmer_mixed_2024} we needed the same bounds but only in \(L^2\), and hence chapter 3 to 5 extend the study of this new operator \(\mathcal{P}_t\). The proof of the main theorem can be found in the last chapter. 

\section*{Acknowledgements}
The author expresses gratitude to Jill Pipher and Martin Dindo\v{s} for many insightful discussions in this topic that helped with refining the applied methods. 

\section{Preliminaries}

\subsection{Notation and Setting}
We will work on the upper half space \(\Omega=\mathbb{R}^{n+1}_+=\mathbb{R}^n\times (0,\infty)\), where we call the last component the \(t\)-direction or component and define
\begin{itemize}
    \item a boundary cube centered at \(x\in\partial\Omega\) with radius \(r\) by \(\Delta_r(x)=\Delta(x,r):=B(x,r)\cap\partial\Omega\);
    \item a Carleson region over a boundary ball as \(T(\Delta_r(x)):=B(x,r)\cap\Omega\);
    \item the matrix \(A(x,t)\in\mathbb{R}^{(n+1)\times(n+1)}\) to be bounded and elliptic, i.e. there exists \(\lambda>0\) with
    \[\lambda|\xi|^2\leq \xi^T A(x,t)\xi\leq |\xi|^2\qquad\textrm{ for every }\xi\in \mathbb{R}^{n+1}, \textrm{ and a.e. }(x,t)\in\Omega.\]
    Here \(A\) is potentially nonsymmetric, and we set the different components as
    \[A(x,t)=\begin{pmatrix} \AP(x,t) & b(x,t)\\c(x,t) & d(x,t)\end{pmatrix},\]
    where \(\AP(x,t)\in\mathbb{R}^{n\times n}, c(x,t)\in\mathbb{R}^{1\times n}, b(x,t)\in\mathbb{R}^{ n\times 1}, d(x,t)\in\mathbb{R}\). All of these components are functions in \((x,t)\).
    \item the operator \((L^t_{||})_{t>0}=(\mathrm{div}_{||}(\AP\nabla_{||} \cdot))_{t>0}\) as a family of elliptic operators.
    \item a nontangential cone as \(\Gamma_\alpha(x):=\{(y,t)\in\Omega; |x-y|<\alpha t\}\)
    with aperture \(\alpha>0\) 
    \item the mean-valued nontangential maximal function \[\tilde{N}_\alpha(F)(x)=\tilde{N}_\alpha(F)(x):=\sup_{(y,s)\in\Gamma_\alpha(x)}\fint_{B(y,s,s/2)}|F(z,\tau)|dzd\tau; \]
\end{itemize}

We set the area function of \(F\in L^2_{loc}\) as
\begin{align}\mathcal{A}(F)(x):=\Big(\int_{\Gamma_\alpha(x)}\frac{|F(x,t)|^2}{t^{n+1}}dxdt\Big)^{1/2},\label{recall:areafunctiondef}\end{align}
the square function as \(S(F)(x):=\mathcal{A}(tF)(x)\),
and the Carleson function as
\begin{align}\mathcal{C}(f)(x):=\Big(\sup_{r>0}\frac{1}{|\Delta(x,r)|}\int_{T(\Delta(x,r))}\frac{|f(y,t)|^2}{t}dydt\Big)^{1/2}\label{recall:carelsonfunctiondef}\end{align}
for \(x\in \partial\Omega=\mathbb{R}^n\), where \(\Delta(x,r)\) is a boudnary ball centered at \(x\in \partial\Omega\) with radius \(r>0\).
\medskip

Furthermore, we call a Borel function \(g:\partial\Omega\to\mathbb{R}\) a \textit{Haj\l{}asz upper gradient} of \(f:\partial\Omega\to\mathbb{R}\) if
\[|f(X)-f(Y)|\leq |X-Y|(g(X)+g(Y))\qquad\textrm{for a.e. }X,Y\in \partial\Omega.\]
We denote the collection of all Haj\l{}asz upper gradients of \(f\) as \(\mathcal{D}(f)\) and define \(\dot{L}_1^p(\partial\Omega)\) by all \(f\) with
\[\Vert f\Vert_{\dot{L}^p_1(\partial\Omega)}:= \inf_{g\in\mathcal{D}(f)}\Vert g\Vert_{L^p(\partial\Omega)}<\infty.\]
This space is also called \textit{homogeneous Haj\l{}asz Sobolev space}. In the case of a flat boundary we have that \(\dot{L}_1^p(\mathbb{R}^n)=\dot{W}^{1,p}(\mathbb{R}^n)\) with comparable norms. 

We can define the regularity problem with boundary data in \(\dot{L}_1^p(\partial\Omega)\).

\begin{defin}[\((RP)_p^L\)]\label{def:L^pregularityProblem}
    We say the \textit{\(L^p\)-regularity boundary value problem} is solvable for \(L\) if for all boundary data for \(f\in C_C(\Omega)\cap \dot{L}^p_{1}(\partial\Omega)\) the unique solution \(u\in W_{loc}^{1,2}(\Omega)\) of
    \[\begin{cases}
        Lu=0&\Omega,
        \\
        u=f&\partial\Omega,
    \end{cases}\]
    satisfies
    \begin{align}\Vert \tilde{N}(\nabla u)\Vert_{L^p(\partial\Omega)}\lesssim\Vert f\Vert_{\dot{L}_{1}^p(\partial\Omega)},\label{eq:RPinequalityElliptic}\end{align}
    where the implied constants are independent of \(u\) and \(f\). In this case we also write that \((RP)^L_p\) holds for \(L\).
\end{defin}

\subsection{The approximation operator \(\mathcal{P}_t\)}
We set the operator \[\mathcal{P}_t:=e^{-t^2L_{||}^t}.\]
Without the dependence of the operator \(L_{||}^t\) on \(t\), this \(\mathcal{P}_t\) is the ellipticized heat semigroup, and solves the heat equation on the upper half space. Since this operator does not satisfy any PDE directly, we can decouple the dependencies in \(t\) and define
\[Wf(x,t,s):=e^{-tL_{||}^s}f(x)\]
as the solution to the ("t-independent") heat equation
\[\begin{cases} \partial_t Wf(x,t,s)-L_{||}^s Wf(x,t,s)=0 &(x,t)\in \Omega, \\ Wf(x,0,s)=f(x) &x\in\partial\Omega\end{cases}\]
for fixed \(s>0\). Taking the partial derivative of \(Wf\) in \(s\) yields
\[\begin{cases} \partial_t \partial_sWf(x,t,s)-L_{||}^s \partial_s Wf(x,t,s)=\mathrm{div}_{||}(\partial_s \AP(x,s)\nabla_{||}Wf(x,t,s)), &(x,t)\in \Omega, \\ Wf(x,0,s)=0, &x\in\partial\Omega.\end{cases}\]
By Duhamel's principle, we can obtain an explicit formula for \(\partial_s Wf(x,t,s)\). Hence we obtain
\begin{align*}
\partial_t\mathcal{P}_tf(x)&=[\partial_tWf(x,t^2,s) + \partial_sWf(x,t^2,s)]|_{s=t}
\\
&=2tL_{||}^te^{-t^2L_{||}^t}f(x) + \int_0^t2\tau e^{(t^2-\tau^2)L_{||}}\mathrm{div}(\partial_t\AP(x,t)\nabla_{||}e^{-\tau^2L_{||}^t}f(x))d\tau
\\
&=:W_1f(x,t) + W_2f(x,t).
\end{align*}
This argument is the same that was already presented in \cite{ulmer_mixed_2024}. Please note that if the operator \(L_{||}^t\) is independent of \(t\), then \(W_2f\equiv 0\).
\medskip

When clear from context and when there is no chance of confusing the present derivative with the full one is, we will drop the subscript \(||\).

\subsection{Kernel bounds and \(L^2-L^2\) off-diagonal estimates}

We have the following bounds of the kernel \(K_t(x,y)\) of the semigroup \(e^{tL_{||}}\).

\begin{prop}[Prop 4.3 in \cite{hofmann_lp_2022} or Theorem 6.17 in  \cite{ouhabaz_analysis_2004}]\label{Prop:KernelBounds}
    For any \(l\in\mathbb{N}\) there exists \(C=C(n,\lambda,l), \beta=\beta(n,\lambda)>0\) such that
    \begin{align}
        |\partial_t^l K_t(x,y)|\leq C_l t^{-\frac{n}{2}-l}e^{-\beta\frac{|x-y|^2}{t}}
    \end{align}
    for all \(x,y\in\mathbb{R}^n\).
\end{prop}

The kernel bounds give rise to the following local bounds
\begin{prop}[Prop 11 in \cite{hofmann_dirichlet_2022} and Cor. 5.6 in \cite{ulmer_mixed_2024}, proof of Lemma 6.4 in \cite{ulmer_mixed_2024}]\label{prop:PROP11}
    Let \(x\in \partial\Omega=\mathbb{R}^n\) and \((y,t)\in\Gamma_\alpha(x)\), then there exists \(C=C(n,\lambda,\alpha)>0\) such that
    \begin{enumerate}[(i)]
        \item \begin{align}
        W_1f(y,t)=tL_{||}e^{-t^2L_{||}^t}f(y,t)\leq C M[\nabla_{||}f](x),
    \end{align}
    \item \begin{align}
        e^{-t^2L_{||}^t}(f-(f)_{\Delta_{\alpha t/2}(x)})(y,t)\leq C M[\nabla_{||}f](x),
    \end{align}
    \item \begin{align}
        \fint_{\Delta_{\alpha t/2}(x)}|\nabla_{||} e^{-t^2L_{||}^t}f(z,\tau)|^2dz d\tau\leq C M[\nabla_{||}f]^2(x)\label{eq:NablaBounded}.
    \end{align}
    \end{enumerate}
\end{prop}

A direct consequence of \refprop{prop:PROP11} is the following:
\begin{cor}\label{cor:NontangentialBoundsOnW_1andW_2}
    It holds for \(f\in L^{p}(\partial \Omega), p>1\)
    \begin{enumerate}[(i)]
        \item \[\Vert \tilde{N}(\nabla_{||}\mathcal{P}_sf)\Vert_{L^p(\partial\Omega)}\lesssim\Vert \nabla_{||} f\Vert_{L^p(\partial\Omega)},\textrm{ and}\]
        \item \[\Vert \tilde{N}(W_1\mathcal{P}_sf)\Vert_{L^p(\partial\Omega)}\lesssim\Vert \nabla_{||} f\Vert_{L^p(\partial\Omega)}.\]
    \end{enumerate}
\end{cor}

Furthermore, we have off-diagonal estimates
\begin{prop}[Off-diagonal estimates]\label{prop:Off-diagonalEstimates}
    For \(T\in \{e^{-tL},\sqrt{t}e^{-tL},t\partial_te^{-tL}\}\) there exists \(C,\alpha>0\) such that
    \[\Vert Th\Vert_{L^2(E)}\lesssim e^{-\frac{\alpha d(E,F)^2}{t}}\Vert h\Vert_{L^2(\mathbb{R}^n)}\] 
    for all \(h\in L^2(\mathbb{R}^n)\) with \(\mathrm{supp}(h)\subset F\).
\end{prop}

\subsection{\(L^2\) Square Function Bounds on terms involving \(\mathcal{P}_t\)}

First, we note that since \(W_1f\) and \(W_2f\) satisfy some PDE, we have a Cacciopolli inequality and Cacciopolli-type inequality respectively.
\begin{prop}[Lemma 4.1 and Lemma 4.2 in \cite{ulmer_mixed_2024}]\label{prop:CacciopolliTypeInequality}
    Let \(B(x,t,4r)\subset \Omega\), then 
    \begin{align}
        \int_{B(x,t,r)} |\nabla_{||}W_1f(y,s)|^2 dyds\lesssim \frac{1}{r^2}\int_{B(x,t,2r)} |W_1f(y,s)|^2 dyds,
    \end{align}
    and
    \begin{align}
        \fint_{B(x,t,r)} |\nabla_{||}W_2f(y,s)|^2 dyds&\lesssim  \frac{1}{r^2}\fint_{B(x,t,2r)} |W_2f(y,s)|^2 dyds\nonumber
        \\
        &\qquad + \sup_{s\in (t-2r,t+2r)}\Vert\partial_s\AP(y,s)\Vert_{L^\infty(\mathbb{R}^n)}^2M[\nabla_{||}f]^2(x).\label{eq:secondsummand}
    \end{align}
\end{prop}

The statement of Lemma 4.2 in \cite{ulmer_mixed_2024} is slightly different to above proposition. However, we can easily see that the second summand in \eqref{eq:secondsummand} results from the trivial \(L^\infty\) bound on \(\partial_t\AP\) and \eqref{eq:NablaBounded}, while we use smoothness of the semi group and the resulting Cacciopolli inequality on time slices for the first term of \eqref{eq:secondsummand}.
\medskip

Additionally, we already have \(L^2\) square function bounds.
\begin{prop}[Lemma 7.7 in \cite{ulmer_mixed_2024}]\label{prop:L^2SquareFctBoundW_1}
For \(f\in W^{1,2}(\mathbb{R}^n)\) it holds that
\begin{align}
    \Vert \mathcal{A}(W_1f)\Vert_{L^2(\mathbb{R}^n)}\lesssim \Vert \nabla f\Vert_{L^2(\mathbb{R}^n)}
\end{align}
and
\begin{align}
    \Vert \mathcal{A}(t\nabla W_1f)\Vert_{L^2(\mathbb{R}^n)}\lesssim \Vert \nabla f\Vert_{L^2(\mathbb{R}^n)}.
\end{align}
\end{prop}
\begin{prop}[Lemma 7.2 and Lemma 7.3 in \cite{ulmer_mixed_2024}]\label{prop:L^2SquareFctBoundW_2}
For \(f\in W^{1,2}(\mathbb{R}^n)\) it holds that
\begin{align}
    \Vert \mathcal{A}(W_2f)\Vert_{L^2(\mathbb{R}^n)}\lesssim \Vert \nabla f\Vert_{L^2(\mathbb{R}^n)}
\end{align}
and
\begin{align}
    \Vert \mathcal{A}(t\nabla W_2f)\Vert_{L^2(\mathbb{R}^n)}\lesssim \Vert \nabla f\Vert_{L^2(\mathbb{R}^n)}.
\end{align}
\end{prop}

\begin{lemma}\label{lemma:L^2SquareFctBoundP-MVP}
    For \(f\in W^{1,2}(\mathbb{R}^n)\) it holds that
    \[\Vert \mathcal{A}(t^{-1}(\mathcal{P}_tf-(\mathcal{P}_tf)_{\Delta(\cdot,t/2)}))\Vert_{L^2}\lesssim \Vert\nabla f\Vert_{L^2(\mathbb{R}^n)}.\]
\end{lemma}

\begin{proof}
    For each cone \(\Gamma_\alpha(x)\) we have by Poincar\'{e}'s inequality
    \begin{align*}
        \int_{\Gamma_\alpha(x)}\frac{|\mathcal{P}_tf-(\mathcal{P}_tf)_{\Delta(x,t/2)}|^2}{t^{n+3}}dydt&=\int_0^\infty \frac{1}{t^{n+3}}\Big(\int_{\Delta(x,t/2)}|\mathcal{P}_tf-(\mathcal{P}_tf)_{\Delta(x,t/2)}|^2dy\Big)dt
        \\
        &\lesssim\int_0^\infty \frac{1}{t^{n+1}}\Big(\int_{\Delta(x,t/2)}|\nabla_{||}\mathcal{P}_tf|^2dy\Big)dt
        \\
        &\lesssim\int_{\Gamma_\alpha(x)}\frac{|\nabla_{||}\mathcal{P}_tf|^2}{t^{n+3}}dydt.
    \end{align*}

    Hence
    \begin{align*}
        &\Vert \mathcal{A}(t^{-1}(\mathcal{P}_tf-(\mathcal{P}_tf)_{\Delta(\cdot,t/2)}))\Vert_{L^2}&
        \\
        &\lesssim \int_\Omega \frac{|\nabla_{||}\mathcal{P}_tf(x)|^2}{t}dxdt\lesssim \int_\Omega \frac{\AP\nabla_{||}\mathcal{P}_tf(x)\cdot\nabla_{||}\mathcal{P}_tf(x)}{t}dxdt
        \\
        &\lesssim \int_\Omega \frac{W_1f(x)\mathcal{P}_t(f-(\mathcal{P}_tf)_{\Delta(x,t/4)})(x)}{t^2}dxdt
        \\
        &\lesssim \Vert \mathcal{A}(W_1f)\Vert_{L^2} + \sigma \int_{\partial\Omega}\Big(\int_{\Gamma_\alpha(x)}\frac{|\mathcal{P}_tf-(\mathcal{P}_tf)_{\Delta(y,t/4)}|^2}{t^{n+3}}dydt\Big) dx
        \\
        &\lesssim \Vert \mathcal{A}(W_1f)\Vert_{L^2} + \sigma\Vert \mathcal{A}(t^{-1}(\mathcal{P}_tf-(\mathcal{P}_tf)_{\Delta(\cdot,t/2)}))\Vert_{L^2}.
    \end{align*}
    By choosing a small constant \(\sigma\) we can hide the last term on the left hand side and finish the proof by invoking \refprop{prop:L^2SquareFctBoundW_1}.
\end{proof}

\section{Bounds for Carleson functions}

Recall the defintion \eqref{recall:carelsonfunctiondef} of the Carleson function for \(x\in\partial\Omega\) by
\begin{align*}
    C(f)(x):=\sup_{\Delta=\Delta(x,r)\subset\partial\Omega} \Big(\frac{1}{\sigma(\Delta)}\int_{T(\Delta)} \frac{|f(x,t)|^2}{t}dxdt\Big)^{\frac{1}{2}}.
\end{align*}
We are going to establish that the Carleson function of \(\partial_t\mathcal{P}_tf\) and \(t\nabla_{||}\partial_t\mathcal{P}_tf\) are in \(L^\infty\) if \(f\in W^{1,\infty}(\mathbb{R}^n)\).

\begin{lemma}\label{lemma:CarlesonFunctionBoundW_1}
    Let \(f\in W^{1,\infty}(\mathbb{R}^n)\), then
    \begin{enumerate}[(i)]
        \item \(\sup_{\Delta\subset\partial\Omega}\frac{1}{\sigma(\Delta)}\int_{T(\Delta)}\frac{|W_1f(x,t)|^2}{t}dxdt\lesssim \Vert\nabla_{||}f\Vert_{L^\infty(\mathbb{R}^n)}\), and\label{item:LocalSquareBoundW1}
        \item \(\sup_{\Delta\subset\partial\Omega}\frac{1}{\sigma(\Delta)}\int_{T(\Delta)}\frac{|t\nabla_{||} W_1f(x,t)|^2}{t}dxdt\lesssim \Vert\nabla_{||}f\Vert_{L^\infty(\mathbb{R}^n)}\).\label{item:LocalSquareBoundNablaW1}
    \end{enumerate}
\end{lemma}

\begin{proof}
    First we prove \eqref{item:LocalSquareBoundW1}. We introduce a smooth cut-off function \(\eta\in C_0^\infty(3\Delta)\) with \(\eta\equiv 1\) on \(2\Delta\) and \(|\nabla \eta|\lesssim \frac{1}{l(\Delta)}\) and split the data \(f\) into a local and far-away part. We obtain

    \begin{align*}
        &\int_{T(\Delta)}\frac{|W_1f(x,t)|^2}{t}dxdt=\int_{T(\Delta)}\frac{|W_1(f-(f)_{3\Delta})(x,t)|^2}{t}dxdt
        \\
        &\qquad \leq \int_{T(\Delta)}\frac{|W_1((f-(f)_{3\Delta})\eta)(x,t)|^2}{t}dxdt
        \\
        &\qquad \qquad+ \int_{T(\Delta)}\frac{|W_1((f-(f)_{3\Delta})(1-\eta))(x,t)|^2}{t}dxdt.
    \end{align*}
    Since the first integral is bounded by 
    \[\Vert\nabla( (f-(f)_{3\Delta})\eta)\Vert_{L^2(\mathbb{R}^n)}\lesssim \frac{1}{l(\Delta)}\Vert (f-(f)_{3\Delta}) \Vert_{L^2(3\Delta)} + \Vert \nabla_{||} f\Vert_{L^2(3\Delta)}\lesssim \Vert \nabla_{||} f\Vert_{L^2(3\Delta)}\]
    by \refprop{prop:L^2SquareFctBoundW_1} and Poincar\'{e}'s inequality, we only need to deal with the second integral.

    \smallskip
    We have for a fixed \(t\in (0,l(\Delta))\) with off-diagonal estimates \refprop{prop:Off-diagonalEstimates} for \(W_1g(x,t)=t\partial_\tau e^{-\tau L^s}g(x)|_{s=t, \tau=t^2}\)
    \begin{align*}
        &\Vert W_1 ((f-(f)_{3\Delta})(1-\eta))\Vert_{L^2(\Delta)}
        \\
        &\lesssim \Vert W_1 (\chi_{4\Delta}(f-(f)_{3\Delta})(1-\eta))\Vert_{L^2(\Delta)}+\sum_{k\geq 2}\Vert W_1 (\chi_{2^{k+1}\Delta\setminus 2^{k}\Delta}(f-(f)_{3\Delta}))\Vert_{L^2(\Delta)}
        \\
        &\lesssim \frac{1}{t}e^{-c\frac{l(\Delta)^2}{t^2}}\Vert f-(f)_{3\Delta} \Vert_{L^2(4\Delta\setminus 2\Delta)} +  \sum_{k\geq 2}\frac{1}{t}e^{-c\frac{2^{2k}l(\Delta)^2}{t^2}}\Vert f-(f)_{3\Delta}\Vert_{L^2(2^{k+1}\Delta\setminus 2^k\Delta)}
        \\
        &\lesssim \frac{1}{t}e^{-c\frac{l(\Delta)^2}{t^2}}\Vert f-(f)_{4\Delta} \Vert_{L^2(4\Delta\setminus 2\Delta)} + \frac{1}{t}e^{-c\frac{l(\Delta)^2}{t^2}}\Vert (f)_{4\Delta}-(f)_{3\Delta} \Vert_{L^2(4\Delta\setminus 2\Delta)}
        \\
        &\qquad +\sum_{k\geq 2}\frac{1}{t}e^{-c\frac{2^{2k}l(\Delta)^2}{t^2}}\Big(\Vert f-(f)_{2^{k+1}\Delta}\Vert_{L^2(2^{k+1}\Delta\setminus 2^k\Delta)} 
        \\
        &\qquad + \sum_{j=2}^{k-1} \Vert (f)_{2^{j+1}\Delta}-(f)_{2^j\Delta}\Vert_{L^2(2^{k+1}\Delta\setminus 2^k\Delta)} + \Vert (f)_{4\Delta}-(f)_{3\Delta}\Vert_{L^2(2^{k+1}\Delta\setminus 2^k\Delta)}\Big).
    \end{align*}

    For each term of the form \(\Vert (f)_{2^{j+1}\Delta}-(f)_{2^j\Delta}\Vert_{L^2(2^{k+1}\Delta\setminus 2^k\Delta)}\) or \(\Vert f-(f)_{4\Delta}\Vert_{L^2(4\Delta)}\) we apply Poincar\'{e}'s inequality and obtain
    \begin{align*}
        &\frac{l(\Delta)}{t}e^{-c\frac{l(\Delta)^2}{t^2}}\Vert \nabla_{||}f \Vert_{L^2(4\Delta\setminus 2\Delta)} + \frac{l(\Delta)}{t}e^{-c\frac{l(\Delta)^2}{t^2}}\Vert\nabla_{||}f \Vert_{L^2(4\Delta\setminus 2\Delta)}
        \\
        &\qquad +\sum_{k\geq 2}\frac{1}{t}e^{-c\frac{2^{2k}l(\Delta)^2}{t^2}}\Big(2^{k+1}l(\Delta)\Vert \nabla_{||}f\Vert_{L^2(2^{k+1}\Delta)}
        \\
        &\qquad + \sum_{j=2}^{k-1} 2^{j+1}l(\Delta)\big\Vert \Vert \nabla_{||} f\Vert_{L^2(2^{j+1}\Delta)}\big\Vert_{L^2(2^{k+1}\Delta\setminus 2^k\Delta)}
        \\
        &\qquad + l(\Delta)\big\Vert \Vert \nabla_{||} f\Vert_{L^2(4\Delta)}\big\Vert_{L^2(2^{k+1}\Delta\setminus 2^k\Delta)}\Big)
        \\
        &\lesssim \frac{l(\Delta)^{\frac{n}{2}+1}}{t}e^{-c\frac{l(\Delta)^2}{t^2}}\Vert \nabla_{||}f \Vert_{\infty} + \frac{l(\Delta)^{\frac{n}{2}+1}}{t}e^{-c\frac{l(\Delta)^2}{t^2}}\Vert\nabla_{||}f \Vert_{\infty}
        \\
        &\qquad +\sum_{k\geq 2}\frac{1}{t}e^{-c\frac{2^{2k}l(\Delta)^2}{t^2}}\Big(2^{(k+1)(\frac{n}{2}+1)}l(\Delta)^{\frac{n}{2}+2}\Vert \nabla_{||}f\Vert_{\infty}
        \\
        &\qquad + \sum_{j=2}^{k-1} 2^{(k+1)(\frac{n}{2}) + j+1}l(\Delta)^{(j+1)(\frac{n}{2}+1)}\Vert \nabla_{||} f\Vert_{\infty} + l(\Delta)^{\frac{n}{2}+2} \Vert \nabla_{||} f\Vert_{\infty}\Big)
        \\
        &\lesssim \frac{l(\Delta)^{\frac{n}{2}+1}}{t}\Big(e^{-c\frac{l(\Delta)^2}{t^2}} + \sum_{k\geq 2} (k+1)2^{(k+1)(\frac{n}{2}+1)}e^{-c\frac{2^{2k}l(\Delta)^2}{t^2}}\Big)\Vert \nabla_{||} f\Vert_{\infty}
        \\
        &\lesssim \frac{l(\Delta)^{\frac{n}{2}+1}}{t}\Big(\sum_{k\geq 1} (k+1)2^{(k+1)(\frac{n}{2}+1)}e^{-c\frac{2^{2k}l(\Delta)^2}{t^2}}\Big)\Vert \nabla_{||} f\Vert_{\infty}
        \\
        &\lesssim \frac{l(\Delta)^{\frac{n}{2}+1}}{t}\Big(\sum_{k\geq 1} 2^{(k+1)(\frac{n}{2}+2)}e^{-c\frac{2^{2k}l(\Delta)^2}{t^2}}\Big)\Vert \nabla_{||} f\Vert_{\infty}.
    \end{align*}

    With this in hand, we return now to 
    \begin{align*}
    &\int_{T(\Delta)}\frac{|W_1((f-(f)_{3\Delta})(1-\eta))(x,t)|^2}{t}dxdt
    \\
    &\qquad\lesssim \int_0^{l(\Delta)}\frac{l(\Delta)^{n+2}}{t^3}\Big(\sum_{k\geq 1} 2^{(k+1)(\frac{n}{2}+2)}e^{-c\frac{2^{2k}l(\Delta)^2}{t^2}}\Big)^2\Vert \nabla_{||} f\Vert_{\infty}^2dt
    \\
    &\qquad \lesssim l(\Delta)^{n+2}\Vert \nabla_{||} f\Vert_{\infty}^2\int_0^{l(\Delta)}\Big(\sum_{k\geq 1} \frac{1}{t^{3/2}}2^{(k+1)(\frac{n}{2}+2)}e^{-c\frac{2^{2k}l(\Delta)^2}{t^2}}\Big)^2dt.
    \end{align*}

    We can study the function \(t\mapsto \frac{1}{t^{3/2}}e^{-c\frac{a^2}{t^2}}\) for \(t\in (0,l(\Delta))\) and observe that it attains its maximum if \(t=l(\Delta)\), whence
    \[\sum_{k\geq 1} \frac{1}{t^{3/2}}2^{(k+1)(\frac{n}{2}+2)}e^{-c\frac{2^{2k}l(\Delta)^2}{t^2}}\lesssim \frac{1}{l(\Delta)^{3/2}}\sum_{k\geq 1} 2^{(k+1)(\frac{n}{2}+2)}e^{-c2^{2k}}.\]
    We consider the sum as Riemann sums of the corresponding integral
    \[\int_0^\infty x^{\frac{n}{2}+1}e^{-cx^2}dx,\]
    which converges. Hence we obtain in total
    
    \begin{align*}
    &\int_{T(\Delta)}\frac{|W_1((f-(f)_{3\Delta})(1-\eta))(x,t)|^2}{t}dxdt
    \\
    &\qquad \lesssim l(\Delta)^{n+2}\Vert \nabla_{||} f\Vert_{\infty}^2\int_0^{l(\Delta)}\frac{1}{l(\Delta)^3}dt=\sigma(\Delta) \Vert \nabla_{||} f\Vert_{\infty}^2.
    \end{align*}
\end{proof}

We can replicate the same proof for a different operator with the same kernel bounds.
\begin{cor}\label{cor:CarlesonFunctionBoundP-MVP}
    Let \(f\in W^{1,\infty}(\mathbb{R}^n)\). Then
    \[\sup_{\Delta\subset\partial\Omega}\frac{1}{\sigma(\Delta)}\int_{T(\Delta)}\frac{|\mathcal{P}_tf(x)-(\mathcal{P}_tf)_{\Delta(x,t/2)}|^2}{t^3}dxdt\lesssim \Vert\nabla_{||}f\Vert_{L^\infty}.\]
\end{cor}

\begin{proof}
    We observe that we only needed the \(L^2\)-square function bound, kernel estimates and off-diagonal estimates for the operator \(te^{-t^2L_{||}^t}\) for a fixed \(t>0\) in the previous proof. However, the operator \(\frac{1}{t}\mathcal{P}_t(f(x)-(\mathcal{P}_tf)_{\Delta(x,t/2)})\) satisfies the same kernel and off-diagonal estimates (see \refprop{Prop:KernelBounds}, \refprop{prop:Off-diagonalEstimates}) and has the same \(L^2\)-square function bound (see \reflemma{lemma:L^2SquareFctBoundP-MVP}). Hence the proof works analogously to the proof of \eqref{item:LocalSquareBoundW1} in \reflemma{lemma:CarlesonFunctionBoundW_1}.
\end{proof}

Next, we obtain the same bounds for \(W_2f\).
\begin{lemma}\label{lemma:CarlesonFunctionBoundW_2}
    Let \(f\in W^{1,\infty}(\mathbb{R}^n)\). Then
    \begin{enumerate}
        \item \(\sup_{\Delta\subset\partial\Omega}\frac{1}{\sigma(\Delta)}\int_{T(\Delta)}\frac{|W_2f(x,t)|^2}{t}dxdt\lesssim \Vert\nabla_{||}f\Vert_{L^\infty}\), and\label{item:LocalSquareBoundW2}
        \item \(\sup_{\Delta\subset\partial\Omega}\frac{1}{\sigma(\Delta)}\int_{T(\Delta)}\frac{|t\nabla_{||} W_2f(x,t)|^2}{t}dxdt\lesssim \Vert\nabla_{||}f\Vert_{L^\infty}\).\label{item:LocalSquareBoundNablaW2}
    \end{enumerate}
\end{lemma}

\begin{proof}
    For \eqref{item:LocalSquareBoundW2} we use Minkowski inequality to obtain

    \begin{align*}
        &\int_{T(\Delta)}\frac{|W_sf|^2}{s}dxds= \int_0^{\infty}\int_{\Delta}\frac{1}{s}\Big|\int_0^{s} \tau e^{-(s^2-\tau^2)L_{||}^s}\mathrm{div}(\partial_s \AP(x,s)\nabla e^{-\tau^2L_{||}^s}f)d\tau\Big|^2 dx ds
        \\
        &\leq \int_0^{l(\Delta)}\frac{1}{s}\Big(\int_0^{s} \Vert\tau e^{-(s^2-\tau^2)L_{||}^s}\mathrm{div}(\partial_s \AP(x,s)\nabla e^{-\tau^2L_{||}^s}f)\Vert_{L^2(\Delta)} d\tau\Big)^2 dx ds.
    \end{align*}
    Let us use off-diagonal estimates \refprop{prop:Off-diagonalEstimates} and cut-off functions \(\eta_k:=\chi_{2^k\Delta\setminus2^{k-1}\Delta}\) such that \(\sum_k\chi_{2^k\Delta\setminus 2^{k-1}\Delta} + \chi_{2\Delta}=1\) to obtain
    \begin{align}
        &\Vert\tau e^{-(s^2-\tau^2)L_{||}^s}\mathrm{div}(\partial_s \AP(x,s)\nabla e^{-\tau^2L_{||}^s}f)\Vert_{L^2(\Delta)}\nonumber
        \\
        &\leq \sum_{k\geq 2}\Vert\tau e^{-(s^2-\tau^2)L_{||}^s}\mathrm{div}(\eta_k\partial_s \AP(x,s)\nabla e^{-\tau^2L_{||}^s}f)\Vert_{L^2(\Delta)}\nonumber
        \\
        &\qquad + \Vert\tau e^{-(s^2-\tau^2)L_{||}^s}\mathrm{div}(\chi_{2\Delta}\partial_s \AP(x,s)\nabla e^{-\tau^2L_{||}^s}f)\Vert_{L^2(\Delta)}\nonumber
        \\
        &\leq \sum_{k\geq 2}\frac{\tau}{\sqrt{s^2-\tau^2}}e^{-c\frac{2^{2k}l(\Delta)^2}{s^2-\tau^2}}\Vert\partial_s \AP(x,s)\nabla e^{-\tau^2L_{||}^s}f\Vert_{L^2(2^k\Delta\setminus 2^{k-1}\Delta)}\nonumber
        \\
        &\qquad + \frac{\tau}{\sqrt{s^2-\tau^2}}\Vert\partial_s \AP(x,s)\nabla e^{-\tau^2L_{||}^s}f\Vert_{L^2(2\Delta)}=I+II.\nonumber
    \end{align}
    We continue for the first term with the same ideas, and get
    \begin{align*}
        I&\leq \sum_{k\geq 2}\sum_{m\geq 2}\frac{\tau}{\sqrt{s^2-\tau^2}}e^{-c\frac{2^{2k}l(\Delta)^2}{s^2-\tau^2}}\Vert\partial_s \AP(x,s)\nabla e^{-\tau^2L_{||}^s}(\eta_{k+m}(f-(f)_{2^k\Delta}))\Vert_{L^2(2^k\Delta\setminus 2^{k-1}\Delta)}\nonumber
        \\
        & \qquad + \sum_{k\geq 2}\frac{\tau}{\sqrt{s^2-\tau^2}}e^{-c\frac{2^{2k}l(\Delta)^2}{s^2-\tau^2}}\Vert\partial_s \AP(x,s)\nabla e^{-\tau^2L_{||}^s}(\chi_{2^{k+1}\Delta}(f-(f)_{2^k\Delta}))\Vert_{L^2(2^k\Delta\setminus 2^{k-1}\Delta)}\nonumber
        \\
        &\leq \sum_{k\geq 2}\sum_{m\geq 2}\frac{1}{\sqrt{s^2-\tau^2}}e^{-c\frac{2^{2k}l(\Delta)^2}{s^2-\tau^2} - c\frac{2^{2m}l(\Delta)^2}{\tau^2}}\Vert\partial_s \AP(\cdot,s)\Vert_\infty\Vert(f-(f)_{2^k\Delta})\Vert_{L^2(2^{k+m}\Delta)}\nonumber
         \\
        & \qquad + \sum_{k\geq 2}\frac{1}{\sqrt{s^2-\tau^2}}e^{-c\frac{2^{2k}l(\Delta)^2}{s^2-\tau^2}}\Vert\partial_s \AP(\cdot,s)\Vert_{L^\infty}\Vert f-(f)_{2^k\Delta}\Vert_{L^2(2^{k+1}\Delta)}.\nonumber
    \end{align*}
    The second one yields
    \begin{align*}
    II&\leq \sum_{m\geq 3}\frac{\tau}{\sqrt{s^2-\tau^2}}\Vert\partial_s \AP(x,s)\nabla e^{-\tau^2L_{||}^s}(\eta_m(f-(f)_{4\Delta}))\Vert_{L^2(2\Delta)}\nonumber
        \\
        &\qquad + \frac{\tau}{\sqrt{s^2-\tau^2}}\Vert\partial_s \AP(x,s)\nabla e^{-\tau^2L_{||}^s}(\chi_{4\Delta}(f-(f)_{4\Delta}))\Vert_{L^2(2\Delta)}\nonumber
        \\
        &\leq \sum_{m\geq 3}\frac{1}{\sqrt{s^2-\tau^2}}e^{-c\frac{2^{2m}l(\Delta)^2}{\tau^2}}\Vert\partial_s \AP(\cdot,s)\Vert_\infty\Vert(f-(f)_{4\Delta})\Vert_{L^2(2^m\Delta)}\nonumber
        \\
        &\qquad + \frac{1}{\sqrt{s^2-\tau^2}}\Vert\partial_s \AP(\cdot,s)\Vert_\infty\Vert(f-(f)_{4\Delta})\Vert_{L^2(4\Delta)}.
    \end{align*}
    By Poincar\'{e}'s inequality, we observe
    \begin{align*}
        &\Vert f-(f)_{2^k\Delta}\Vert_{L^2(2^{k+m}\Delta)}
        \\
        &\qquad\leq \Vert f-(f)_{2^{k+m}\Delta}\Vert_{L^2(2^{k+m}\Delta)} + \sum_{l=1}^m \Vert(f)_{2^{k+l}\Delta}-(f)_{2^{k+l-1}\Delta}\Vert_{L^2(2^{k+m}\Delta)}
        \\
        &\qquad\lesssim 2^{(k+m)}l(\Delta)\Vert \nabla_{||}f\Vert_{L^2(2^{k+m}\Delta)} + \sum_{l=1}^m 2^{k+l}l(\Delta)\Vert\nabla_{||}f\Vert_{L^2(2^{k+m}\Delta)}
        \\
        &\qquad\lesssim 2^{(k+m)}l(\Delta)\Vert \nabla_{||}f\Vert_{L^2}\lesssim  2^{(k+m)(\frac{n}{2}+1)}l(\Delta)|\Delta|^{1/2}\Vert \nabla_{||}f\Vert_{L^\infty},
    \end{align*}
    and hence since \(\tau,\sqrt{s^2-\tau^2}\leq s\leq l(\Delta)\)
    \begin{align*}
        I+II&\lesssim\Vert\partial_s \AP(\cdot,s)\Vert_\infty\Vert\nabla_{||}f\Vert_{L^\infty}|\Delta|^{1/2}\Big(\sum_{k\geq 2}\sum_{m\geq 2}\frac{2^{(k+m)(\frac{n}{2}+1)}l(\Delta)}{\sqrt{s^2-\tau^2}}e^{-c\frac{2^{2k}l(\Delta)^2}{s^2-\tau^2} - c\frac{2^{2m}l(\Delta)^2}{\tau^2}}
         \\
        & \qquad + \sum_{k\geq 2}\frac{\tau 2^{(k+1)\frac{n}{2}}}{\sqrt{s^2-\tau^2}}e^{-c\frac{2^{2k}l(\Delta)^2}{s^2-\tau^2}} + \sum_{m\geq 3}\frac{2^{(m+1)(\frac{n}{2}+1)}l(\Delta)}{\sqrt{s^2-\tau^2}}e^{-c\frac{2^{2m}l(\Delta)^2}{\tau^2}} + \frac{\tau}{\sqrt{s^2-\tau^2}}\Big).
    \end{align*}
    We consider the sums as Riemann sums of integral. For instance, the second term corresponds to the integral
    \[\int_0^\infty\frac{\tau}{\sqrt{s^2-\tau^2}}x^{\frac{n}{2}-1}e^{-cx^2\frac{l(\Delta)^2}{s^2-\tau^2}}dx,\]
    which after the change of variables \(y=\frac{l(\Delta)^2}{s^2-\tau^2}x\) gives a convergent integral
    \[\frac{\tau\sqrt{s^2-\tau^2}}{l(\Delta)\sqrt{s^2-\tau^2}}\int_0^\infty y^{\frac{n}{2}-1}e^{-y^2}dy\lesssim C.\]
    Similar arguments for the other two sums lead to the conclusion
    \[I+II\lesssim \Vert\partial_s \AP(\cdot,s)\Vert_\infty\Vert\nabla_{||}f\Vert_{L^\infty}|\Delta|^{1/2}\Big(1+ \frac{\tau}{\sqrt{s^2-\tau^2}}\Big).\]
    We can now complete the proof of \eqref{item:LocalSquareBoundW2} with the condition \eqref{cond:mixedL1LInftyCond} and
    \begin{align*}
        &\int_0^{l(\Delta)}\frac{1}{s}\Big(\int_0^s \big(1+\frac{\tau}{\sqrt{s^2-\tau^2}}\big)\Vert\partial_s \AP(\cdot,s)\Vert_\infty\Vert\nabla_{||}f\Vert_{L^\infty}|\Delta|^{1/2} d\tau\Big)^2ds
        \\
        &\qquad \lesssim \Vert\nabla_{||}f\Vert_{L^\infty}^2|\Delta|\int_0^{l(\Delta)} \Vert\partial_s \AP(\cdot,s)\Vert_\infty^2sds\lesssim \Vert\nabla_{||}f\Vert_{L^\infty}^2|\Delta|.
    \end{align*}

    For \eqref{item:LocalSquareBoundNablaW2} we have by Cacciopolli type inequality \refprop{prop:CacciopolliTypeInequality}
    \begin{align*}
        \int_{T(\Delta)}\frac{|t\nabla W_2f(x,t)|^2}{t}dxdt&\lesssim \int_{T(2\Delta)}\frac{|W_2f(x,t)|^2}{t}dxdt
        \\
        &\qquad + |\Delta|\int_0^{2l(\Delta)}\Vert\partial_t A(\cdot,t)\Vert_{L^\infty(\mathbb{R}^n)}^2t\Vert\nabla_{||}f\Vert_{L^\infty(\mathbb{R}^n)}^2dt
        \\
        &\lesssim |\Delta| \Vert \nabla_{||}f\Vert_{L^\infty(\mathbb{R}^n)}.
    \end{align*}
\end{proof}

\section{Square function bounds on Hardy atoms}

\begin{defin}
    Let \(1<\beta<\infty\). We call a function \(a:\mathbb{R}^n\to\mathbb{R}\) a \textit{homogeneous Hardy-Sobolev \(\beta\)-atom} associated to a boundary ball \(\Delta\subset\mathbb{R}^n\) if
    \begin{enumerate}[(i)]
        \item \(\mathrm{supp}(a)\subset \Delta\); and
        \item \(\Vert \nabla a\Vert_{L^\beta(\mathbb{R}^n)}\leq |\Delta|^{-\frac{1}{\beta}} \).
    \end{enumerate}
    If \(f\) can be written as
    \begin{align}f=\sum_{j=1}^\infty\lambda_ja_j\label{eq:AtomicHardyDecomposition}\end{align}
    for \(\beta\)-atoms \(a_j\) and coefficients \(\lambda_j\in\mathbb{R}\) with \(\sum_{j=1}^\infty|\lambda_j|<\infty\), we say that \(f\in \dot{HS}^{1,\beta}_{atom}\), where \(\Vert f\Vert_{\dot{HS}^{1,\beta}_{atom}}:=\inf\sum_{j=1}^\infty|\lambda_j|\) with an infimum that is take over all choices of decompositions \eqref{eq:AtomicHardyDecomposition}. 
\end{defin}

\begin{lemma}\label{lemma:HardyNormBounds}
    Let \(f\in \dot{HS}_{atom}^{1,2}(\mathbb{R}^n)\). Then
    \begin{enumerate}
        \item \(\Vert\mathcal{A}(\partial_t\mathcal{P}f)\Vert_{L^1(\mathbb{R}^n)}\lesssim \Vert \nabla_{||}f\Vert_{\dot{HS}^{1,2}(\mathbb{R}^n)}\)\label{item:HardySquareBound}
        \item \(\Vert\mathcal{A}(t\nabla \partial_t\mathcal{P}f)\Vert_{L^1(\mathbb{R}^n)}\lesssim \Vert \nabla_{||}f\Vert_{\dot{HS}_{atom}^{1,2}(\mathbb{R}^n)}\)\label{item:HardySquareBoundNabla}
    \end{enumerate}
\end{lemma}

\begin{proof}
    First, we note that it is enough to show
    \[\Vert\mathcal{A}(\partial_t\mathcal{P}f)\Vert_{L^1(\mathbb{R}^n)}, \Vert\mathcal{A}(t\nabla \partial_t\mathcal{P}f)\Vert_{L^1(\mathbb{R}^n)}\leq C\]
    for all homogeneous Hardy-Sobolev \(1/2\)-atoms \(f\) associated with \(\Delta\), whence we assume that \(f\) is such an atom going forward.
    
    We begin with showing \eqref{item:HardySquareBound}. We split the integral into a local and a far away part 
    \begin{align*}
        \Vert\mathcal{A}(\partial_t\mathcal{P}f)\Vert_{L^1(\mathbb{R}^n)}=\Vert\mathcal{A}(\partial_t\mathcal{P}f)\Vert_{L^1(\mathbb{R}^n\setminus 5\Delta)} + \Vert\mathcal{A}(\partial_t\mathcal{P}f)\Vert_{L^1(5\Delta)}.
    \end{align*}
    For the local part  we have by Hölder's inequality, \refprop{prop:L^2SquareFctBoundW_1}, \refprop{prop:L^2SquareFctBoundW_2} and the properties of the Hardy space that
    \begin{align*}
        \Vert\mathcal{A}(\partial_t\mathcal{P}f)\Vert_{L^1(5\Delta)}\lesssim \Vert\mathcal{A}(\partial_t\mathcal{P}f)\Vert_{L^2(5\Delta)}|\Delta|^{1/2}\lesssim \Vert\nabla_{||} f\Vert_{L^2(\mathbb{R}^n)}|\Delta|^{1/2}\lesssim \Vert \nabla_{||}f\Vert_{H^1(\mathbb{R}^n)}.
    \end{align*}
    For the away part, we split into
    \begin{align*}
        \Vert\mathcal{A}(\partial_t\mathcal{P}f)\Vert_{L^1(\mathbb{R}^n\setminus 5\Delta)}&=\int_{\partial\Omega\setminus 5\Delta }\Big(\int_{\Gamma(y)}\frac{|W_1(f)|^2}{t^{n+1}}dxdt\Big)^{1/2}dy
        \\
        &\qquad + \int_{\partial\Omega\setminus 5\Delta}\Big(\int_{\Gamma(y)}\frac{|W_2(f)|^2}{t^{n+1}}dxdt\Big)^{1/2}dy
        \\
        &=I_1+I_2.
    \end{align*}

    First, we note that the kernel bounds in \refprop{Prop:KernelBounds} imply
    \begin{align}
        \int_{\Delta_{t/2}(y)}|tL^te^{-t^2L^t}f(x,t)|^2dx\lesssim \begin{cases} \frac{1}{t^{2n+2}}e^{-c\frac{\mathrm{dist}(\Delta_{t/2}(y),\Delta)^2}{t^2}}\Vert f\Vert_{L^1(\Delta)}^2 &\textrm{if }\mathrm{dist}(\Delta_{t/2}(y),\Delta)>0, \\ \frac{1}{t^{2n+2}}\Vert f\Vert_{L^1(\Delta)}^2& \mathrm{else}.\end{cases}\label{eq:CasesForL^1BoundofTermInProof}
    \end{align}
    Making use of \eqref{eq:CasesForL^1BoundofTermInProof}, we observe that if \(\mathrm{dist}(y,\Delta)\approx 2^jl(\Delta)\) then
    \begin{align*}
        &\int_{\Gamma(y)}\frac{|W_1(f)|^2}{t^{n+1}}dxdt=\int_0^\infty \frac{1}{t^{n+1}}\int_{\Delta_{t/2}(y)}|tL^te^{-t^2L^t}f(x,t)|^2dx dt
        \\
        &\qquad=\int_0^{2^{j-1}l(\Delta)} \frac{e^{-2\frac{|2^{j}l(\Delta)-\frac{t}{2}|^2}{t^2}}}{t^{2n+3}}\Vert f\Vert_{L^1(\mathbb{R}^n)}^2dx dt + \int_{2^{j-1}l(\Delta)}^\infty \frac{1}{t^{2n+3}}\Vert f\Vert_{L^1(\mathbb{R}^n)}^2 dt
        \\
        &\qquad\lesssim\int_0^{2^{j-1}l(\Delta)} \frac{l(\Delta)^{n+2}e^{-\frac{2^{2(j-1)}l(\Delta)^2}{t^2}}}{t^{2n+3}}\Vert \nabla_{||}f\Vert_{L^2(\mathbb{R}^n)}^2dx dt
        \\
        &\qquad\qquad + \int_{2^{j-1}l(\Delta)}^\infty \frac{l(\Delta)^{n+2}}{t^{2n+3}}\Vert \nabla_{||}f\Vert_{L^2(\mathbb{R}^n)}^2 dt
        \\
        &\qquad \lesssim \frac{2^{-j(2n+2)}}{l(\Delta)^{n}}\Vert \nabla_{||}f\Vert_{L^2(\mathbb{R}^n)}^2\lesssim \frac{2^{-j(2n+2)}}{l(\Delta)^{2n}}.
    \end{align*}
    Summing over annuli, we obtain
    \begin{align*}
        I_1&\lesssim \sum_{j\geq 3}\int_{2^j\Delta\setminus 2^{j-1}\Delta }\Big(\int_{\Gamma(y)}\frac{|W_1(f)|^2}{t^{n+1}}dxdt\Big)^{1/2}dy \lesssim \sum_{j\geq 3}\int_{2^j\Delta\setminus 2^{j-1}\Delta }\frac{2^{-j(n+1)}}{l(\Delta)^{n}}dy 
        \\
        &\lesssim \sum_{j\geq 3} 2^{-j}\lesssim C.
    \end{align*}

    Next, let us bound \(I_2\).
    To start with we fix \(t\) and we look at the term 
    \[\int_{\Delta_{t/2}(y)}\Big|\int_0^t2\tau e^{-(t^2-\tau^2)L^t}\mathrm{div}(\partial_t\AP\nabla e^{-\tau^2L^t}f)d\tau\Big|^2dx.\]

    First assume that \(0\leq t\leq 2^{l-1}l(\Delta)\). Then \(\Delta_{t/2}(y)\subset \Delta_{2^{l-1}l(\Delta)}(y):=\tilde{\Delta}\), and we have for a dualising function \(g\in L^2(\tilde{\Delta})\) 
    \begin{align}
        &\Vert\tau e^{-(t^2-\tau^2)L_{||}^t}\mathrm{div}(\partial_t \AP(x,t)\nabla e^{-\tau^2L_{||}^t}f)\Vert_{L^2(\tilde{\Delta})}\nonumber
        \\
        &=\tau\int_{\mathbb{R}}\nabla e^{-(t^2-\tau^2)L_{||}^t} g \cdot\partial_t \AP(x,t)\nabla e^{-\tau^2L_{||}^t}f dx\nonumber
        \\
        &\lesssim
            \tau\Vert\partial_t\AP\Vert_\infty\big(\Vert \nabla e^{-(t^2-\tau^2)L_{||}^t} g\Vert_{L^2(2\tilde{\Delta})} \Vert \nabla e^{-\tau^2L_{||}^t}f\Vert_{L^2(2\tilde{\Delta})}\nonumber
        \\
        &\qquad+ \Vert \nabla e^{-(t^2-\tau^2)L_{||}^t} g\Vert_{L^2(2\Delta)} \Vert \nabla e^{-\tau^2L_{||}^t}f\Vert_{L^2(2\Delta)}\nonumber
        \\
        &\qquad+ \Vert \nabla e^{-(t^2-\tau^2)L_{||}^t} g\Vert_{L^2(\mathbb{R}^n\setminus 2\tilde{\Delta})} \Vert \nabla e^{-\tau^2L_{||}^t}f\Vert_{L^2(\mathbb{R}^n\setminus 2\Delta)}\big).\label{eq:ProofUsefulEquation1}
    \end{align}
    Let us note that by using a smooth cut-off function \(\eta\) with \(\mathrm{supp}(\eta)\subset\frac{3}{2}\tilde{\Delta}\) and \(\eta\equiv 1\) on \(\tilde{\Delta}\) we have
    \[\int_{\tilde{\Delta}}|\nabla e^{-\tau^2L_{||}^t}f|^2dx\lesssim \int_{2\tilde{\Delta}} W_1f \cdot \mathcal{P}_tf\eta^2 dx + \int_{2\tilde{\Delta}} \nabla e^{-\tau^2L_{||}^t}f\cdot\nabla \eta \mathcal{P}_tf\eta dx\]
    which implies \(\int_{\tilde{\Delta}}|\nabla e^{-\tau^2L_{||}^t}f|^2dx\lesssim \int_{\tilde{\Delta}}|W_1f|^2dx + \frac{1}{t^2}\int_{\tilde{\Delta}}|\mathcal{P}_tf|^2dx\). This observation holds not only for the set \(\tilde{\Delta}\), but also for integrals over other sets and enlargements thereof.
    For each of those terms we can use a similar observation to \eqref{eq:CasesForL^1BoundofTermInProof} and bound \(\eqref{eq:ProofUsefulEquation1}\) by
    \begin{align*}
        &\Vert\partial_t\AP\Vert_\infty\Big[\Big(\int_{2\tilde{\Delta}}\big(\frac{\tau}{\sqrt{t^2-\tau^2}\tau^{n+1}}e^{-\frac{2^{2(j-1)} l(\Delta)^2}{\tau^2}}\Vert f\Vert_{L^1(\Delta)}\big)^2dx\Big)^{1/2}\Vert g\Vert_{L^2}
        \\
        &\qquad +  \Big(\int_{2\Delta}\big(\frac{1}{\sqrt{t^2-\tau^2}^{n+1}}e^{-\frac{2^{2(j-1)} l(\Delta)^2}{t^2-\tau^2}}\Vert g\Vert_{L^1(\tilde{\Delta})}\big)^2dx\Big)^{1/2}\Vert f\Vert_{L^2}
        \\
        &\qquad +  \tau\Big(\int_{\mathbb{R}^n\setminus 2\tilde{\Delta}}\big(\frac{1}{\sqrt{t^2-\tau^2}^{n+1}}e^{-\frac{\mathrm{dist}(x,\tilde{\Delta})^2}{t^2-\tau^2}}\Vert g\Vert_{L^1(\tilde{\Delta})}\big)^2dx\Big)^{1/2} 
        \\
        &\qquad\qquad \cdot\Big(\int_{\mathbb{R}^n\setminus 2\Delta}\big(\frac{1}{\tau^{n+1}}e^{-\frac{\mathrm{dist}(x,\Delta)^2}{\tau}}\Vert f\Vert_{L^1(\Delta)}\big)^2dx\Big)^{1/2} \Big].
    \end{align*}
    Further, we can use Hölder's inequality and Poincar\'{e}'s inequality and continue with
    \begin{align*}
        &\Vert\partial_t\AP\Vert_\infty\Big[\frac{\tau 2^{jn/2}l(\Delta)^n}{\sqrt{t^2-\tau^2}\tau^{n+1}}e^{-\frac{2^{2(j-1)} l(\Delta)^2}{\tau^2}}\Vert f\Vert_{L^2(\Delta)}\Vert g\Vert_{L^2}
        \\
        &\qquad +  \frac{2^{jn/2}l(\Delta)^n}{\sqrt{t^2-\tau^2}^{n+1}}e^{-\frac{2^{2(j-1)} l(\Delta)^2}{t^2-\tau^2}}\Vert g\Vert_{L^2(\tilde{\Delta})}\Vert f\Vert_{L^2}
        \\
        &\qquad + \frac{2^{jn/2}l(\Delta)^{n/2}}{\sqrt{t^2-\tau^2}^{n+1}}\Vert g\Vert_{L^2(\tilde{\Delta})}\Big(\int_{2^{j-1}l(\Delta)}^\infty r^{n-1}e^{-\frac{r^2}{t^2-\tau^2}}dr\Big)^{1/2}
        \\
        &\qquad\qquad\cdot\frac{l(\Delta)^{n/2}}{\tau^{n}}\Vert f\Vert_{L^2(\Delta)}\Big(\int_{l(\Delta)}^\infty r^{n-1}e^{-\frac{r^2}{\tau^2}}dr\Big)^{1/2} \Big]
        \\
        &\lesssim \Vert\partial_t\AP\Vert_\infty\Vert \nabla_{||}f\Vert_{L^2(\Delta)}\Vert g\Vert_{L^2}\Big[\frac{\tau 2^{jn/2}l(\Delta)^{n+1}}{\sqrt{t^2-\tau^2}\tau^{n+1}}e^{-\frac{2^{2(j-1)} l(\Delta)^2}{\tau^2}}  +  \frac{2^{jn/2}l(\Delta)^{n+1}}{\sqrt{t^2-\tau^2}^{n+1}}e^{-\frac{2^{2(j-1)} l(\Delta)^2}{t^2-\tau^2}}
        \\
        &\qquad + \frac{2^{jn/2}l(\Delta)^{n+1}}{\sqrt{t^2-\tau^2}^{n/2+1}\tau^{n/2}}\Big(\int_{2^{j-1}l(\Delta)/\sqrt{t^2-\tau^2}}^\infty r^{n-1}e^{-r^2}dr\Big)^{1/2}\Big(\int_{2^{l(\Delta)}/\tau}^\infty r^{n-1}e^{-r^2}dr\Big)^{1/2}\Big].
    \end{align*}
    The integral expressions are a little bit delicate. If \(n\) is even, than \(\int_{s}^\infty r^{n-1}e^{-r^2}dr=P(a)e^{-a^2}\) where \(P\) is a polynomial of degree \(n-2\). If \(n\) is odd, we can bound the integral above by the same expression with \(n+1\) instead, since we only integrate over numbers greater than \(1\). Hence, we can bound the first integral by \(\frac{2^{(j-1)(n-2)}l(\Delta)^{n-2}}{\sqrt{t^2-\tau^2}^{n-2}}\) or \(\frac{2^{(j-1)(n-1)}l(\Delta)^{n-1}}{\sqrt{t^2-\tau^2}^{n-1}}\). Since \(\frac{2^{j-1}l(\Delta)}{\sqrt{t^2-\tau^2}}\lesssim 1\), we can continue to only look at the first case without loss of generality.
    We obtain
    \begin{align*}
        &\lesssim \Vert\partial_t\AP\Vert_\infty\Vert\nabla_{||}f\Vert_{L^2}\Big[\frac{\tau 2^{jn/2}l(\Delta)^{n+1}}{\sqrt{t^2-\tau^2}\tau^{n+1}}e^{-\frac{2^{2(j-1)} l(\Delta)^2}{\tau^2}} +  \frac{2^{jn/2}l(\Delta)^{n+1}}{\sqrt{t^2-\tau^2}^{n+1}}e^{-\frac{2^{2(j-1)} l(\Delta)^2}{t^2-\tau^2}}
        \\
        &\qquad + \frac{2^{j(n-1)}l(\Delta)^{2n-1}}{\sqrt{t^2-\tau^2}^{n}\tau^{n-1}}e^{-\frac{2^{2(j-1)} l(\Delta)^2}{t^2-\tau^2}} \Big]
        \\
        &\lesssim \Vert\partial_t\AP\Vert_\infty\Big(\frac{l(\Delta)^{n/2+1}}{\sqrt{t^2-\tau^2}\tau^{n}}e^{-\frac{2^{2(j-1)} l(\Delta)^2}{\tau^2}} +\frac{ l(\Delta)^{n/2+1}}{\sqrt{t^2-\tau^2}^{n+1}}e^{-\frac{2^{2(j-1)} l(\Delta)^2}{t^2-\tau^2}}
        \\
        &\qquad + \frac{2^{j(n-1)}l(\Delta)^{3n/2-1}}{\sqrt{t^2-\tau^2}^{n}\tau^{n-1}}e^{-\frac{2^{2(j-1)} l(\Delta)^2}{t^2-\tau^2}}\Big).
    \end{align*}
    In the last step we used that \(f\) is a Hardy-Sobolev atom.
    \medskip
    
    Now, we can estimate the integral over a truncated cone for a point \(y\in\partial\Omega\) with \(\mathrm{dist}(y,\Delta)\approx 2^jl(\Delta)\). We have
    \begin{align}
        &\Big(\int_{\Gamma^{2^jl(\Delta)}(y)}\frac{|W_2(f)|^2}{t^{n+1}}dxdt\Big)^{1/2}\nonumber
        \\
        &=\Big(\int_0^{2^{j-1}l(\Delta)} \frac{1}{t^{n+1}}\int_{\Delta_{t/2}(y)}\Big|\int_0^t2\tau e^{-(t^2-\tau^2)L^t}\mathrm{div}(\partial_t\AP\nabla e^{-\tau^2L^t}f)d\tau\Big|^2dx dt\Big)^{1/2}\nonumber
        \\
        &\lesssim \Big(\int_0^{2^{j-1}l(\Delta)} \frac{\Vert\partial_t\AP\Vert_\infty^2}{t^{n+1}}\Big(\int_0^t\frac{2^{jn/2}l(\Delta)^{n/2+1}}{\sqrt{t^2-\tau^2}\tau^{n}}e^{-\frac{2^{2(j-1)} l(\Delta)^2}{\tau^2}}\nonumber
        \\
        &\qquad +\frac{2^{jn/2}l(\Delta)^{n/2+1}}{\sqrt{t^2-\tau^2}^{n+1}}e^{-\frac{2^{2(j-1)} l(\Delta)^2}{t^2-\tau^2}} + \frac{2^{j(n-1)}l(\Delta)^{3n/2-1}}{\sqrt{t^2-\tau^2}^{n}\tau^{n-1}}e^{-\frac{2^{2(j-1)} l(\Delta)^2}{t^2-\tau^2}}d\tau\Big)^2 dt\Big)^{1/2}\nonumber
        \\
        &\lesssim \Big(\int_0^{2^{j-1}l(\Delta)} \Vert\partial_t\AP\Vert_\infty^2t\Big(\frac{2^{jn}l(\Delta)^{n+2}}{t^{3n+2}} + \frac{2^{2j(n-1)}l(\Delta)^{3n-2}}{t^{5n-2}}\Big)
        \\
        &\qquad\qquad \cdot e^{-2\frac{2^{2(j-1)} l(\Delta)^2}{t^2}}\Big(\int_0^t\frac{1}{\sqrt{t^2-\tau^2}}d\tau\Big)^2 dt\Big)^{1/2}\nonumber
        \\
        &\lesssim \Big(\int_0^{2^{j-1}l(\Delta)}\Vert\partial_t\AP\Vert_\infty^2t \Big(\frac{1}{2^{j(2n+2)}l(\Delta)^{n}} + \frac{1}{2^{3jn}l(\Delta)^{n}}\Big) dt\Big)^{1/2}\nonumber
        \\
        &\lesssim \Big(\frac{1}{2^{j(n+1)}}+\frac{1}{2^{j(n+n/2)}}\Big)\frac{1}{l(\Delta)^{n}}\Big(\int_0^{2^{j-1}l(\Delta)} \Vert\partial_t\AP\Vert_\infty^2 t dt\Big)^{1/2}\nonumber
        \\
        &\lesssim \frac{1}{2^{j(n+1/2)}}\frac{1}{l(\Delta)^{n}}.\label{eq:Proofshortconeheight}
    \end{align}

    On the other hand we have 
    
    \begin{align}
        &\Big(\int_{\Gamma(y)\setminus\{t\leq 2^{j-1}l(\Delta)\}}\frac{|W_2(f)|^2}{t^{n+1}}dxdt\Big)^{1/2}\nonumber
        \\
        &=\Big(\int_{2^{j-1}l(\Delta)}^\infty \frac{1}{t^{n+1}}\int_{\Delta_{t/2}(y)}\Big|\int_0^t2\tau e^{-(t^2-\tau^2)L^t}\mathrm{div}(\partial_t\AP\nabla e^{-\tau^2L^t}f)d\tau\Big|^2dx dt\Big)^{1/2}\nonumber
        \\
        &\lesssim\Big(\int_{2^{j-1}l(\Delta)}^\infty \frac{1}{t^{n+1}}\Big(\int_0^t2\tau \Vert e^{-(t^2-\tau^2)L^t}\mathrm{div}(\partial_t\AP\nabla e^{-\tau^2L^t}f)\Vert_{L^2(\Delta_{t/2}(y))}d\tau\Big)^2 dt\Big)^{1/2}\nonumber
        \\
        &\lesssim\Big(\int_{2^{j-1}l(\Delta)}^\infty \frac{1}{t^{n+1}}\Big(\int_0^{t/2}\tau \Vert e^{-(t^2-\tau^2)L^t}\mathrm{div}(\partial_t\AP\nabla e^{-\tau^2L^t}f)\Vert_{L^2(\Delta_{t/2}(y))}d\tau\Big)^2 dt\Big)^{1/2}\nonumber
        \\
        &\qquad+ \Big(\int_{2^{j-1}l(\Delta)}^\infty \frac{1}{t^{n+1}}\Big(\int_{t/2}^t\tau \Vert e^{-(t^2-\tau^2)L^t}\mathrm{div}(\partial_t\AP\nabla e^{-\tau^2L^t}f)\Vert_{L^2(\Delta_{t/2}(y))}d\tau\Big)^2 dt\Big)^{1/2}.\label{eq:NonlocalConeEstimateW2}
    \end{align}

    For \(t/2\leq \tau\leq t\) we have
    \begin{align*}
        \tau\Vert e^{-(t^2-\tau^2)L^t}\mathrm{div}(\partial_t\AP\nabla e^{-\tau^2L^t}f)\Vert_{L^2(\Delta_{t/2}(y))}&\lesssim \frac{\tau}{\sqrt{t^2-\tau^2}}\Vert\partial_t A\Vert_{\infty}\Vert\nabla e^{-\tau^2L^t}f\Vert_{L^2(\mathbb{R}^n)}
        \\
        &\lesssim \frac{\tau}{\sqrt{t^2-\tau^2}\tau^{n+1}}\Vert\partial_t A\Vert_{\infty}\Vert f\Vert_{L^1(\mathbb{R}^n)}
        \\
        &\lesssim  \frac{1}{\sqrt{t^2-\tau^2}t^{n}}\Vert\partial_t A\Vert_{\infty}\Vert f\Vert_{L^1(\mathbb{R}^n)}
        \\
        &\lesssim  \frac{l(\Delta)^{n/2+1}}{\sqrt{t^2-\tau^2}t^{n}}\Vert\partial_t A\Vert_{\infty}\Vert \nabla_{||}f\Vert_{L^2(\mathbb{R}^n)}
        \\
        &\lesssim  \frac{l(\Delta)}{\sqrt{t^2-\tau^2}t^{n/2}}\Vert\partial_t A\Vert_{\infty}.
    \end{align*}

    And for \(0\leq \tau\leq t/2\) we have for a dualising function \(g\in L^2(\Delta_{t/2}(y))\)
    \begin{align*}
        &\tau\Vert e^{-(t^2-\tau^2)L^t}\mathrm{div}(\partial_t\AP\nabla e^{-\tau^2L^t}f)\Vert_{L^2(\Delta_{t/2}(y))}
        \\
        &\lesssim \tau\Vert\partial_t A\Vert_{\infty}\Big(\Vert \nabla e^{-(t^2-\tau^2)L^t}g \Vert_{L^2(2\Delta)}\Vert\nabla e^{-\tau^2L^t}f\Vert_{L^2(2\Delta)}
        \\
        &\qquad\qquad + \Vert \nabla e^{-(t^2-\tau^2)L^t}g \Vert_{L^2(\mathbb{R}^n)}\Vert\nabla e^{-\tau^2L^t}f\Vert_{L^2(\mathbb{R}^n\setminus 2\Delta)}\Big)
        \\
        &\lesssim \tau\Vert\partial_t A\Vert_{\infty}\Big(\frac{l(\Delta)^{n/2}}{\sqrt{t^2-\tau^2}^{n+1}}\Vert g\Vert_{L^1(\Delta_{t/2}(y))}\Vert \nabla_{||} f\Vert_{L^2(\mathbb{R}^n)}
        \\
        &\qquad\qquad+ \frac{1}{\sqrt{t^2-\tau^2}\tau^{n+1}}\Vert g \Vert_{L^2}\Vert f\Vert_{L^1(\Delta)}\big(\int_{l(\Delta)}^\infty r^{n-1}e^{-\frac{r^2}{\tau^2}}dr\big)^{1/2}\Big)
        \\
        &\lesssim \tau\Vert\partial_t A\Vert_{\infty}\Big(\frac{l(\Delta)^{n/2}}{\sqrt{t^2-\tau^2}^{n+1}}\Vert g\Vert_{L^1(\Delta_{t/2}(y))}\Vert \nabla_{||} f\Vert_{L^2(\mathbb{R}^n)}
        \\
        &\qquad\qquad+ \frac{l(\Delta)^{n/2-1}}{\sqrt{t^2-\tau^2}\tau^{n}}\Vert g \Vert_{L^2}\Vert f\Vert_{L^1(\Delta)}e^{-c\frac{l(\Delta)^2}{\tau^2}}\Big)
        \\
        &\lesssim \tau\Vert\partial_t A\Vert_{\infty}\Big(\frac{l(\Delta)^{n/2}t^{n/2}}{\sqrt{t^2-\tau^2}^{n+1}}\Vert g\Vert_{L^2}\Vert \nabla_{||} f\Vert_{L^2(\mathbb{R}^n)} + \frac{l(\Delta)^{n}}{\sqrt{t^2-\tau^2}\tau^{n}}\Vert g \Vert_{L^2}\Vert \nabla_{||}f\Vert_{L^2(\Delta)}e^{-c\frac{l(\Delta)^2}{\tau^2}}\Big)
        \\
        &\lesssim \tau\Vert\partial_t A\Vert_{\infty}\Big(\frac{l(\Delta)^{n/2}t^{n/2}}{\sqrt{t^2-\tau^2}^{n+1}}\Vert \nabla_{||} f\Vert_{L^2(\mathbb{R}^n)} + \frac{l(\Delta)^n}{\sqrt{t^2-\tau^2}t^n}\Vert \nabla_{||}f\Vert_{L^2(\Delta)}\Big)
        \\
        &\lesssim  \frac{l(\Delta)^{n/2+1}}{\sqrt{t^2-\tau^2} t^{n/2}}\Vert\partial_t A\Vert_{\infty}\Vert\nabla_{||} f\Vert_{L^2(\mathbb{R}^n)}\lesssim \frac{l(\Delta)}{\sqrt{t^2-\tau^2} t^{n/2}}\Vert\partial_t A\Vert_{\infty}.
    \end{align*}
    Hence, we can combine these two observations to continue \eqref{eq:NonlocalConeEstimateW2} with
    \begin{align*}
        &\Big(\int_{2^{j-1}l(\Delta)}^\infty \frac{1}{t^{n+1}}\Big(\int_0^{t/2}\tau \Vert e^{-(t^2-\tau^2)L^t}\mathrm{div}(\partial_t\AP\nabla e^{-\tau^2L^t}f)\Vert_{L^2(\Delta_{t/2}(y))}d\tau\Big)^2 dt\Big)^{1/2}
        \\
        &\qquad+ \Big(\int_{2^{j-1}l(\Delta)}^\infty \frac{1}{t^{n+1}}\Big(\int_{t/2}^t\tau \Vert e^{-(t^2-\tau^2)L^t}\mathrm{div}(\partial_t\AP\nabla e^{-\tau^2L^t}f)\Vert_{L^2(\Delta_{t/2}(y))}d\tau\Big)^2 dt\Big)^{1/2}
        \\
        &\lesssim\Big(\int_{2^{j-1}l(\Delta)}^\infty \frac{l(\Delta)^{2}}{t^{2n+1}}\Vert\partial_t A\Vert_\infty^2\Big(\int_0^{t}\frac{1}{\sqrt{t^2-\tau^2}} d\tau\Big)^2 dt\Big)^{1/2}
        \\
        &\lesssim\Big(\int_{2^{j-1}l(\Delta)}^\infty \Vert\partial_t A\Vert_\infty^2 t dt\Big)^{1/2} \frac{1}{2^{j(n+1)}l(\Delta)^{n}}
        \\
        &\lesssim \frac{1}{2^{j(n+1)}l(\Delta)^{n}},
    \end{align*}
    whence together with \eqref{eq:Proofshortconeheight}
    \begin{align*}
        \Vert \mathcal{A}(W_2(f))\Vert_{L^1(\mathbb{R}^n\setminus 5\Delta)}&\lesssim \sum_{j\geq 2} \big(2^{j+1}l(\Delta)\big)^n\frac{1}{2^{j(n+1/2)}l(\Delta)^n}\leq C.
    \end{align*}

    Lastly, \eqref{item:HardySquareBoundNabla} follows from \eqref{item:HardySquareBound} and \refprop{prop:CacciopolliTypeInequality}.
\end{proof}

Again, above proof gives the corollary for a different operator satisfying the same bounds.
\begin{lemma}\label{lemma:HardyBoundP-MVP}
    Let \(f\in \dot{HS}_{atom}^{1,2}(\mathbb{R}^n)\). Then
    \[\Vert \mathcal{A}(t^{-1}(\mathcal{P}_tf-(\mathcal{P}_tf)_{\Delta(\cdot,t/2)}))\Vert_{L^1(\mathbb{R}^n)} \lesssim \Vert \nabla_{||}f\Vert_{\dot{HS}^{1,2}(\mathbb{R}^n)}.\]
\end{lemma}

\begin{proof}
    Again, we can follow the proof of \eqref{item:HardySquareBound} in \reflemma{lemma:HardyNormBounds} analogously. More precisely, the local part follows as in \eqref{item:HardySquareBound} by the \(L^2\) square function bound \reflemma{lemma:L^2SquareFctBoundP-MVP} and the nonlocal part has an analogous bound to \(I_1\) and uses that \(te^{-t^2L_{||}^t}\) and \(\frac{1}{t}\mathcal{P}_t\) satisfy the same kernel bounds (see \refprop{Prop:KernelBounds}).
\end{proof}

\section{Tent spaces and real interpolation}

Let \(A_0,A_1\) be two normed vector spaces of functions \(a:\mathbb{R}^n\to\mathbb{R}\). For each \(a\in A_0 + A_1\) we define the \(K-\)functional of real interpolation by
\[K(a,t,A_0,A_1)=\inf_{a=a_0+a_1}\Vert a_0\Vert_{A_0} + t\Vert a_1\Vert_{A_1}.\]
For \(0<\theta<1,1\leq q\leq \infty\), we denote by \((A_0,A_1)_{\theta,q}\) the real interpolation space between \(A_0\) and \(A_1\) defined as
\[(A_0,A_1)_{\theta,q}=\Big\{a\in A_1+A_0:\Vert a\Vert_{\theta,q}=\Big(\int_0^\infty (t^{-\theta}K(a,t,A_0,A_1))^q\frac{\mathrm{d}t}{t}\Big)^{1/q}<\infty \Big\}.\]
According to Theorem 3.1.2 in \cite{bergh_interpolation_1976}, \(K\) can be seen as an exact interpolation functor, which means that if an operator \(T\) is bounded from \(A_0\to B_0\) and from \(A_1\to B_1\) for linear normed vector spaces of functions, then
\(T:(A_0,A_1)_{\theta,q}\to (B_0,B_1)_{\theta,q}\) is a bounded linear operator with \(\Vert T\Vert\leq C\Vert T\Vert^{1-\theta}_{A_0\to B_0}\Vert T\Vert^{\theta}_{A_1\to B_1}\).  

Now, we have the following real interpolation result.
\begin{prop}[Thm 0.4 in \cite{badr_abstract_2010}]
    For every \(\beta\in(1,\infty]\) and \(1<p<\infty\) the real interpolation space is
    \[(\dot{HS}^{1,\beta}_{atom}, \dot{W}^{1,\infty})_{1-1/p,p}=\dot{W}^{1,p}.\]
\end{prop}

On the other hand, recall the definition of the area function \eqref{recall:areafunctiondef} and of the Carleson function \eqref{recall:carelsonfunctiondef} and define the tent spaces over \(\mathbb{R}^n\) with parameter \(1\leq p<\infty\) as
\begin{align*}
T^{p,2}(\Omega):=\{F\in L^2_{loc};\Vert F\Vert_{T^{p,2}}:=\Vert \mathcal{A}(F)\Vert_{L^p(\partial\Omega)}<\infty\}
\end{align*}
and
\begin{align*}
T^{\infty,2}(\Omega):=\{F\in L^2_{loc};\Vert F\Vert_{T^{\infty,2}}:=\Vert \mathcal{C}(F)\Vert_{L^\infty(\partial\Omega)}<\infty\}.
\end{align*}

By \cite{coifman_new_1985} we have
\begin{prop}[Theorem 4' in \cite{coifman_new_1985}]
    For every \(1<p<\infty\) the real interpolation space is
    \[(T^{1,2},T^{\infty,2})_{1-1/p,p}=T^{p,2}.\]
\end{prop}

Since for \(1<p<\infty\) the two spaces \(\dot{W}^{1,p}(\mathbb{R}^n)\) and \(\dot{L}^{p}_1(\mathbb{R}^n)\) are the same modulo constants with comparable norms, \reflemma{lemma:CarlesonFunctionBoundW_1}, \reflemma{lemma:CarlesonFunctionBoundW_2}, and \reflemma{lemma:HardyNormBounds} give rise to the following corollary.

\begin{cor}\label{cor:AreaFctBoundsInLp}
    For every \(f\in \dot{L}^{p}_1(\partial\Omega)\) we have
    \[\Vert \mathcal{A}(\partial_t\mathcal{P}_tf)\Vert_{L^p(\partial\Omega)}, \Vert \mathcal{A}(t\nabla_{||}\partial_t\mathcal{P}_tf)\Vert_{L^p(\partial\Omega)}, \Vert \mathcal{A}(tL_{||}\mathcal{P}_tf)\Vert_{L^p(\partial\Omega)}\leq C\Vert \nabla_{||} f\Vert_{L^p(\partial\Omega)}.\]
\end{cor}

In addition to that, the real interpolation method implies via \refcor{cor:CarlesonFunctionBoundP-MVP} and \refcor{lemma:HardyBoundP-MVP} also
\begin{cor}\label{cor:AreaFctBoundsInLp2}
    For every \(f\in \dot{L}^{p}_1(\partial\Omega)\) we have
    \[\Big\Vert \mathcal{A}\Big(\frac{\mathcal{P}_tf-(\mathcal{P}_tf)_{\Delta(\cdot,t/2)}}{t}\Big)\Big\Vert_{L^p(\partial\Omega)}\leq C\Vert \nabla_{||} f\Vert_{L^p(\partial\Omega)}.\]
\end{cor}

\section{Proof of \refthm{MAINTHM}}

Let \(h\) be a quasi dualising function given by the following lemma:
\begin{lemma}[Lemma 2.8 in \cite{kenig_neumann_1995}]\label{lemma:DefAndPropertiesOfH}
    There exists a function \(h:\mathbb{R}^{n+1}\to\mathbb{R}^{n+1}\) such that
    \[\Vert \tilde{N}(\nabla u)\Vert_{L^p(\partial\Omega)}\lesssim \int_\Omega \nabla u \cdot h dxdt,\]
    and
    \[ \int_\Omega F \cdot h dxdt\lesssim \Vert \tilde{N}(F)\Vert_{L^p(\partial\Omega)} \]
    for every vector valued function \(F:\mathbb{R}^{n+1}\to\mathbb{R}^{n+1}\).
\end{lemma}
Now let \(v\) be the solution to the Poisson problem of the adjoint operator
\begin{align}\begin{cases} L^*v=\mathrm{div}(h) &\textrm{in }\Omega,\\ v=0 &\textrm{on }\partial\Omega.\end{cases}\label{PDEForv}\end{align}
For this \(v\) we have the following bounds
\begin{prop}[\cite{kenig_neumann_1995}] 
    Let \(v\) be given by \eqref{PDEForv}. If the elliptic measure \(\omega\in A_\infty(\sigma)\) then there exists some \(1<r<\infty\) such that
    \begin{align}\Vert S(v)\Vert_{L^r(\partial\Omega)},\Vert \tilde{N}(v)\Vert_{L^r(\partial\Omega)}, \Vert \tilde{N}(\delta|\nabla v|)\Vert_{L^r(\partial\Omega)}\lesssim C.\label{eq:BoundsOnV}\end{align}
\end{prop}
Since \cite{ulmer_mixed_2024} showed that \(\omega\in A_\infty(\sigma)\) holds under our assumptions, there exists some \(p>1\) such that the bounds \eqref{eq:BoundsOnV} hold for the dual exponent \(r=p'\).

Then we have for this \(p\)
\begin{align}
    \Vert \tilde{N}(\nabla u)\Vert_{L^p}&\lesssim \int_\Omega \nabla u \cdot h dxdt= \int_\Omega A\nabla u\cdot \nabla v dxdt + \int_{\partial\Omega}A^*\nabla v \cdot u\nu dx\nonumber
    \\
    &=\int_{\partial \Omega}u b\cdot\nabla_{||}v + ud\partial_t v dx\nonumber
    \\
    &=\int_{\partial \Omega}u(x,0) \Big(\int_0^\infty \partial_t\big(b(x,t)\cdot\nabla_{||}v(x,t) + d(x,t)\partial_t v(x,t)\big) dt\Big) dx\label{eq:mainProof1}
\end{align}
Without loss of generality we can approximate all involved components of the matrix \(A\) and \(h\) by smooth functions so that the PDE \eqref{PDEForv} can be used pointwise. The following arguments are independent of this approximation and hence we can take the limit in the very end, but we omit this argument here. Since \(\int_0^\infty \partial_th(x,t)dt=0\)  we can continue 

\begin{align*}
    \eqref{eq:mainProof1}&=\int_{\partial \Omega}u \Big(\int_0^\infty \mathrm{div}_{||}\big(\AP^*(x,t)\cdot\nabla_{||}v(x,t) + c(x,t)\partial_t v(x,t)\big)dt
    \\
    &\qquad + \int_0^\infty\mathrm{div}_{||}(h(x,t)) + \partial_t h(x,t) dt\Big) dx
    \\
    &=\int_{\partial \Omega}\nabla_{||}u(x,0) \Big(\int_0^\infty\AP(x,t)\cdot\nabla_{||}v(x,t) + c(x,t)\partial_t v(x,t)dt\Big)dx
    \\
    &\qquad + \int_{\partial \Omega}\nabla_{||}u(x,0)\cdot h(x,t)dxdt.
\end{align*}

The last term yields by \reflemma{lemma:DefAndPropertiesOfH}
\begin{align*}
    \int_{\partial \Omega}\nabla_{||}u(x,0)\cdot h(x,t)dxdt\lesssim \Vert M[\nabla_{||} f]\Vert_{L^p}\lesssim \Vert \nabla_{||} f\Vert_{L^p}. 
\end{align*}
For the first term we can set 
\[V(x,s):=\int_s^\infty \AP(x,t)\cdot\nabla_{||}v(x,t) + c(x,t)\partial_t v(x,t)dt,\]
and get
\begin{align*}
    &\int_{\partial \Omega}\nabla_{||}u(x,0) \Big(\int_0^\infty\AP(x,t)\cdot\nabla_{||}v(x,t) + c(x,t)\partial_t v(x,t)dt\Big)dx
    \\
    &\qquad=\int_{\partial\Omega}\nabla_{||}f(x) V(x,0) dx \lesssim \int_{\Omega}\partial_s(\nabla_{||}\mathcal{P}_sf(x) V(x,s))dxds
    \\
    &\qquad\lesssim \int_{\Omega}\partial_s\nabla_{||}\mathcal{P}_sf(x)\cdot V(x,s) + \nabla_{||}\mathcal{P}_sf(x)\cdot \partial_s V(x,s) dxds
    \\
    &\qquad\lesssim \int_{\Omega}\partial_{ss}\nabla_{||}\mathcal{P}_sf(x)\cdot V(x,s)sdxds + \int_{\Omega}\nabla_{||}\partial_s\mathcal{P}_sf(x) \cdot\partial_s V(x,s)sdxds
    \\
    &\qquad \qquad + \int_{\Omega}\nabla_{||}\mathcal{P}_sf(x) \cdot\partial_{s} V(x,s)dxds
    \\
    &\qquad =:I+II+III.
\end{align*}
First, we observe that \(|\partial_s V(x,s)|\lesssim |\nabla v(x,s)|\) and hence
\begin{align*}
    II&\lesssim \Vert \mathcal{A}(s\partial_s\nabla_{||}\mathcal{P}_sf)\Vert_{L^p}\Vert S(v)\Vert_{L^{p'}}\lesssim \Vert \nabla_{||}f\Vert_{L^p},
\end{align*}
where we used \refcor{cor:AreaFctBoundsInLp} and \eqref{eq:BoundsOnV}.

\medskip
Next, we have
\begin{align*}
    III&=\int_{\Omega}\nabla_{||}\mathcal{P}_sf(x) \cdot \partial_s(A_{||}(x,s)\nabla_{||}v(x,s) + c(x,s)\partial_s v(x,s))s dxds
    \\
    &\lesssim \int_{\Omega}\nabla_{||}\mathcal{P}_sf \cdot \partial_sA_{||}(x,s)\nabla_{||}v(x,s) s +L^s_{||}\mathcal{P}_sf(x) \partial_s v(x,s)sdxds
    \\
    &\qquad+ \int_{\Omega}\partial_s\nabla_{||}\mathcal{P}_sf \cdot c(x,s)\partial_s v(x,s)s + \nabla_{||}\mathcal{P}_sf \cdot c(x,s)\partial_s v(x,s) dxds
    \\
    &\lesssim \Vert \mathcal{A}(s\partial_s \AP\nabla_{||}\mathcal{P}_sf)\Vert_{L^p}\Vert S(v)\Vert_{L^{p'}} + \Vert \mathcal{A}(sL^s_{||}\mathcal{P}_sf)\Vert_{L^p}\Vert S(v)\Vert_{L^{p'}}
    \\
    &\qquad + \Vert \mathcal{A}(s\partial_s\nabla_{||}\mathcal{P}_sf)\Vert_{L^p}\Vert S(v)\Vert_{L^{p'}} + \Vert \mathcal{A}(\nabla_{||}\mathcal{P}_sf)\Vert_{L^p}\Vert S(v)\Vert_{L^{p'}}.
\end{align*}

To bound all these expressions, we combine \refcor{cor:AreaFctBoundsInLp}, \eqref{eq:BoundsOnV} and \eqref{cond:mixedL1LInftyCond}. For the second and third term we immediately get a bound of \(\Vert\nabla_{||}f\Vert_{L^p}\).
For the first term however, we can make use of \eqref{eq:NablaBounded} and obtain that
\[\mathcal{A}(s\partial_s\AP\nabla_{||}\mathcal{P}_sf)(x)\lesssim M[\nabla_{||}f](x)\Big(\int_0^\infty\Vert \partial_s \AP(\cdot,s)\Vert_{L^\infty(\mathbb{R}^n)}^2sds\Big)^{1/2},\]
whence with \eqref{cond:mixedL1LInftyCond} we obtain the same bound \(\Vert\nabla_{||}f\Vert_{L^p}\). 

And for the last term we observe that
\begin{align*} 
\Vert \mathcal{A}(\nabla_{||}\mathcal{P}_sf)\Vert_{L^p}^p&\leq \Vert \mathcal{A}_2(W_1f)\Vert_{L^p} + \Vert \mathcal{A}_2(\mathcal{P}_sf-(\mathcal{P}_sf)_{\Delta(y,s/2)})\Vert_{L^p}\lesssim \Vert\nabla_{||}f\Vert_{L^p},
\end{align*}
where the aperture of the area functions get doubled. However, since the \(L^p\) norms of area functions with different apertures are comparable, we obtain the last inequality with \refcor{cor:AreaFctBoundsInLp2}.
\medskip

Lastly, for \(I\)
\begin{align*}
    I&=\int_{\Omega}s\nabla_{||}\partial_{ss}\mathcal{P}_sf(x)\cdot\Big(\int_s^\infty \AP(x,t)\nabla_{||}v(x,t) + c(x,t)\partial_tv(x,t)dt\Big) dx
    \\
    &=\int_{\Omega}\nabla_{||}\Big(\int_0^ts\partial_{ss}\mathcal{P}_sf(x)ds\Big)\cdot(\AP(x,t)\nabla_{||}v(x,t) + c(x,t)\partial_tv(x,t)) dxdt
    \\
    &=\int_{\Omega}\nabla_{||}\big(t\partial_{t}\mathcal{P}_tf(x)-\mathcal{P}_tf(x)+f(x)\big)\cdot\big(\AP(x,t)\nabla_{||}v(x,t) + c(x,t)\partial_tv(x,t)\big) dxdt
    \\
    &\lesssim \Vert\mathcal{A}(t\nabla_{||}\partial_t \mathcal{P}_tf)\Vert_{L^p}\Vert S(v)\Vert_{L^{p'}}
    \\
    &\qquad + \Big|\int_{\Omega}\nabla_{||}\big(\mathcal{P}_tf(x)-f(x)\big)\cdot\big(\AP(x,t)\nabla_{||}v(x,t) + c(x,t)\partial_tv(x,t)\big) dxdt\Big|.
\end{align*}
For the last term we have by using the PDE \eqref{PDEForv} and that \((\mathcal{P}f-f)|_{\partial\Omega}=0\)
\begin{align*}
    &\int_{\Omega}\nabla_{||}\big(\mathcal{P}_tf(x)-f(x)\big)\cdot\big(\AP(x,t)\nabla_{||}v(x,t) + c(x,t)\partial_tv(x,t)\big) dxdt
    \\
    &\qquad= \int_{\Omega}\partial_t\mathcal{P}_tf(x)\cdot\big(b(x,t)\nabla_{||}v(x,t) + d(x,t)\partial_tv(x,t)\big)
    \\
    &\qquad \qquad + (\nabla_{||}\mathcal{P}_tf(x) - \nabla_{||}f, \partial_t \mathcal{P}_tf(x))\cdot h(x,t) dxdt
    \\
    &\qquad \lesssim  \Vert\mathcal{A}(\partial_t \mathcal{P}_tf)\Vert_{L^p}\Vert S(v)\Vert_{L^{p'}} +  \Vert\tilde{N}(\partial_t \mathcal{P}_tf)\Vert_{L^p}+\Vert\tilde{N}(\nabla_{||} \mathcal{P}_tf)\Vert_{L^p}\lesssim \Vert \nabla_{||}f\Vert_{L^p}.
\end{align*}
Here we used again the established bounds in \refcor{cor:AreaFctBoundsInLp}, \eqref{eq:BoundsOnV}, \eqref{eq:NablaBounded}, \refcor{cor:NontangentialBoundsOnW_1andW_2} and that \(\Vert\tilde{N}(\partial_t \mathcal{P}_tf)\Vert_{L^p}\lesssim \Vert\mathcal{A}(\partial_t\mathcal{P}_tf)\Vert_{L^p}\). This finishes the proof of \refthm{MAINTHM}.

\bibliographystyle{alpha}
\bibliography{references} 

\newcommand{\etalchar}[1]{$^{#1}$}
\begin{thebibliography}{HKMP15b}

\bibitem[AEN18]{auscher_dirichlet_2018}
Pascal Auscher, Moritz Egert, and Kaj Nyström.
\newblock The {Dirichlet} problem for second order parabolic operators in divergence form.
\newblock {\em Journal de l’École polytechnique — Mathématiques}, 5:407--441, 2018.

\bibitem[AHL{\etalchar{+}}02]{Auscher_Kato}
Pascal Auscher, Steve Hofmann, Michael Lacey, Alan McIntosh, and Ph. Tchamitchian.
\newblock The solution of the kato square root problem for second order elliptic operators on rn.
\newblock {\em Annals of Mathematics}, 156(2):633--654, 2002.

\bibitem[BB10]{badr_abstract_2010}
Nadine Badr and Frédéric Bernicot.
\newblock Abstract {Hardy}–{Sobolev} spaces and interpolation.
\newblock {\em Journal of Functional Analysis}, 259(5):1169--1208, September 2010.

\bibitem[BL76]{bergh_interpolation_1976}
Jöran Bergh and Jörgen Löfström.
\newblock {\em Interpolation {Spaces}: {An} {Introduction}}, volume 223 of {\em Grundlehren der mathematischen {Wissenschaften}}.
\newblock Springer Berlin Heidelberg, Berlin, Heidelberg, 1976.

\bibitem[CFK81]{caffarelli_completely_1981}
Luis~A. Caffarelli, Eugene~B. Fabes, and Carlos~E. Kenig.
\newblock Completely {Singular} {Elliptic}-{Harmonic} {Measures}.
\newblock {\em Indiana University Mathematics Journal}, 30(6):917--924, 1981.
\newblock Publisher: Indiana University Mathematics Department.

\bibitem[CMS85]{coifman_new_1985}
R.~R Coifman, Y~Meyer, and E.~M Stein.
\newblock Some new function spaces and their applications to {Harmonic} analysis.
\newblock {\em Journal of Functional Analysis}, 62(2):304--335, June 1985.

\bibitem[CNS16]{castro_boundedness_2016}
Alejandro~J. Castro, Kaj Nyström, and Olow Sande.
\newblock Boundedness of single layer potentials associated to divergence form parabolic equations with complex coefficients.
\newblock {\em Calculus of Variations and Partial Differential Equations}, 55(5):124, September 2016.

\bibitem[Dah77]{dahlberg_estimates_1977}
Björn E.~J. Dahlberg.
\newblock Estimates of harmonic measure.
\newblock {\em Archive for Rational Mechanics and Analysis}, 65(3):275--288, September 1977.

\bibitem[Dah79]{dahlbert_poisson_1979}
Björn Dahlbert.
\newblock On the {Poisson} integral for {Lipschitz} and \${C}{\textasciicircum}\{1\}\$-domains.
\newblock {\em Studia Mathematica}, 66:13--24, 1979.
\newblock Publisher: Instytut Matematyczny Polskiej Akademii Nauk.

\bibitem[Dah86]{dahlberg_absolute_1986}
Björn E.~J. Dahlberg.
\newblock On the {Absolute} {Continuity} of {Elliptic} {Measures}.
\newblock {\em American Journal of Mathematics}, 108(5):1119--1138, 1986.
\newblock Publisher: Johns Hopkins University Press.

\bibitem[DP22]{dindos_boundary_2022}
Martin Dindoš and Jill Pipher.
\newblock Boundary value problems for elliptic operators satisfying {Carleson} condition, December 2022.
\newblock arXiv:2210.17499 [math].

\bibitem[FJK84]{fabes_necessary_1984}
Eugene~B. Fabes, David~S. Jerison, and Carlos~E. Kenig.
\newblock Necessary and {Sufficient} {Conditions} for {Absolute} {Continuity} of {Elliptic}- {Harmonic} {Measure}.
\newblock {\em Annals of Mathematics}, 119(1):121--141, 1984.
\newblock Publisher: Annals of Mathematics.

\bibitem[HKMP15a]{hofmann_regularity_2015}
Steve Hofmann, Carlos Kenig, Svitlana Mayboroda, and Jill Pipher.
\newblock The regularity problem for second order elliptic operators with complex-valued bounded measurable coefficients.
\newblock {\em Mathematische Annalen}, 361(3):863--907, April 2015.

\bibitem[HKMP15b]{hofmann_square_2015}
Steve Hofmann, Carlos Kenig, Svitlana Mayboroda, and Jill Pipher.
\newblock Square {Function}/{Non}-{Tangential} {Maximal} {Function} {Estimates} and the {Dirichlet} {Problem} for {Non}-{Symmetric} {Elliptic} {Operators}.
\newblock {\em Journal of the American Mathematical Society}, 28(2):483--529, 2015.
\newblock Publisher: American Mathematical Society.

\bibitem[HLMP22a]{hofmann_dirichlet_2022}
Steve Hofmann, Linhan Li, Svitlana Mayboroda, and Jill Pipher.
\newblock The {Dirichlet} problem for elliptic operators having a {BMO} anti-symmetric part.
\newblock {\em Mathematische Annalen}, 382(1):103--168, February 2022.

\bibitem[HLMP22b]{hofmann_lp_2022}
Steve Hofmann, Linhan Li, Svitlana Mayboroda, and Jill Pipher.
\newblock \${L}{\textasciicircum}p\$ theory for the square roots and square functions of elliptic operators having a {BMO} anti-symmetric part.
\newblock {\em Mathematische Zeitschrift}, 301(1):935--976, May 2022.
\newblock arXiv:1908.01030 [math].

\bibitem[JK81]{jerison_dirichlet_1981}
David~S. Jerison and Carlos~E. Kenig.
\newblock The {Dirichlet} {Problem} in {Non}-{Smooth} {Domains}.
\newblock {\em Annals of Mathematics}, 113(2):367--382, 1981.
\newblock Publisher: Annals of Mathematics.

\bibitem[KKPT00]{kenig_new_2000}
C.~Kenig, H.~Koch, J.~Pipher, and T.~Toro.
\newblock A {New} {Approach} to {Absolute} {Continuity} of {Elliptic} {Measure}, with {Applications} to {Non}-symmetric {Equations}.
\newblock {\em Advances in Mathematics}, 153(2):231--298, August 2000.

\bibitem[KP95]{kenig_neumann_1995}
Carlos~E. Kenig and Jill Pipher.
\newblock The {Neumann} problem for elliptic equations with nonsmooth coefficients: {Part} {II}.
\newblock {\em Duke Mathematical Journal}, 81(1), January 1995.

\bibitem[KR08]{kenig_regularity_2008}
Carlos~E. Kenig and David~J. Rule.
\newblock The regularity and {Neumann} problem for non-symmetric elliptic operators.
\newblock {\em Transactions of the American Mathematical Society}, 361(01):125--160, August 2008.

\bibitem[MM80]{modica_construction_1980}
Luciano Modica and Stefano Mortola.
\newblock Construction of a singular elliptic-harmonic measure.
\newblock {\em Manuscripta Mathematica}, 33(1):81--98, March 1980.

\bibitem[Nys17]{nystrom_l2_2017}
Kaj Nyström.
\newblock L2 {Solvability} of boundary value problems for divergence form parabolic equations with complex coefficients.
\newblock {\em Journal of Differential Equations}, 262(3):2808--2939, February 2017.

\bibitem[Ouh04]{ouhabaz_analysis_2004}
El-Maati Ouhabaz.
\newblock {\em Analysis of {Heat} {Equations} on {Domains}. ({LMS}-31)}.
\newblock Princeton University Press, Princeton, UNITED STATES, 2004.

\bibitem[Ulm24]{ulmer_mixed_2024}
Martin Ulmer.
\newblock On a mixed \${L}{\textasciicircum}1-{L}{\textasciicircum}{\textbackslash}infty\$ type {Carleson} condition on the nontangential derivative of {A} for an elliptic operator, January 2024.
\newblock arXiv:2311.00614 [math].

\bibitem[Ulm25]{ulmer_solvability_2025}
Martin Ulmer.
\newblock Solvability of the {Dirichlet} problem using a weaker {Carleson} condition in the upper half plane, March 2025.
\newblock arXiv:2503.19106 [math].

\end{thebibliography}
\end{document}